\documentclass[10pt,reqno]{amsart}
\usepackage{latexsym,amsmath,amsfonts,amscd,amssymb,color,mathtools}
\usepackage{graphicx,wrapfig}
\usepackage[all]{xy}
\definecolor{cobalt}{RGB}{61,89,171}
\usepackage[colorlinks,citecolor=cobalt,linkcolor=cobalt,urlcolor=cobalt,pdfpagemode=UseNone,backref = page]{hyperref}
\usepackage{tabu}
\usepackage{soul}
\usepackage{mathrsfs}

\newcommand{\la}{\langle}
\newcommand{\ra}{\rangle}

\newcommand{\cC}{{\mathcal{C}}}

\newcommand{\cQ}{\mathcal{Q}}

\newcommand{\CC}{\mathbb{C}}

\newcommand{\GG}{\mathbb{G}}

\newcommand{\RR}{\mathbb{R}}

\newcommand{\PP}{\mathbb P}

\newcommand{\im}{\operatorname{im}}

\newcommand{\GL}{\operatorname{GL}}
\newcommand{\PGL}{\operatorname{PGL}}

\newcommand{\ii}{\operatorname{i}}

\DeclareMathOperator{\Aut}{Aut}

\DeclareMathOperator{\rank}{rank}

\renewcommand{\a}{\alpha}
\renewcommand{\b}{\beta}

\newcommand{\g}{\gamma}

\newcommand{\f}{\varphi}

\renewcommand{\l}{\lambda}

\newcommand{\p}{\phi}

\renewcommand{\L}{\Lambda}

\newcommand{\kk}{\mathrm{k}}

\newcommand{\frg}{\mathfrak{g}}

\newcommand{\tr}[1]{\textrm{#1}}

\newtheorem{theorem}{Theorem}
\newtheorem{proposition}[theorem]{Proposition}
\newtheorem{lemma}[theorem]{Lemma}
\newtheorem{corollary}[theorem]{Corollary}

\theoremstyle{definition}
\newtheorem{definition}[theorem]{Definition}

\newtheorem{remark}[theorem]{Remark}

\title{8-dimensional 2-step nilpotent Lie algebras over algebraically closed fields with char $\ne 2,3$.}

\author[G. Bazzoni]{Giovanni Bazzoni}
\address{Dipartimento di Scienza ed Alta Tecnologia, Università degli Studi dell'Insubria, Via Valleggio 11, 22100, Como, Italy}
\email{giovanni.bazzoni@uninsubria.it}

\author[J. Rojo]{Juan Rojo}
\address{ETS Ingenieros Informáticos, Universidad Politécnica de Madrid,
Campus de Montegancedo, 28660, Madrid, Spain}
\email{juan.rojo.carulli@upm.es}


\begin{document}

\begin{abstract}
We provide a self contained, elementary, and geometrically-flavored classification of $8$-dimensional $2$-step nilpotent Lie algebras over algebraically closed fields of characteristic $\ne 2,3$, using the algebro-geometric arguments from \cite{B} and elementary linear algebra.
\end{abstract}

\maketitle



\section{Introduction}

The classification problem stands as one of the main challenges in the theory of (finite dimensional) Lie algebras. 
Over an algebraically closed field, Levi's decomposition theorem asserts that any such Lie algebra is a semidirect product of a solvable Lie algebra and a semisimple Lie algebra, so we are left with the task of classifying both semisimple and solvable Lie algebras. The semisimple Lie algebras are well understood and classified. 
However, the study of solvable lie algebras is harder. Even the simpler case of classifying nilpotent Lie algebras in arbitrary dimension turns out to be a hopeless problem. More precisely, this belongs to a class called \emph{wild problems}, see for instance \cite{Belitskii}. Roughly speaking, a problem is wild if it contains the problem of classifying conjugacy classes of pairs of matrices, i.e.~to find a simultaneous canonical form for pairs of endomorphisms of a vector space. This family of problems is considered to be extremely difficult, and there is no hope whatsoever of finding algorithms that solve them.

Nevertheless, the classification of nilpotent Lie algebras in low dimensions is possible, and it is indeed an interesting problem both in algebra and geometry. The survey \cite{Kaygorodov_2024} contains up-to-date results for the classification of many types of algebras in low dimensions. In the context of differential geometry (the domain of expertise of the authors) nilpotent Lie algebras are interesting because they are closely related to nilmanifolds. Nilmanifolds are compact quotients of a nilpotent Lie group by a subgroup, and they provide an interesting source of examples of closed manifolds. Moreover, one can define a tensor on the Lie algebra, and then extend it to an invariant tensor on the nilmanifold. This means, basically, that in the context of a nilmanifold we can reduce differential geometry to linear algebra. This is extremely useful for computational purposes and has been extensively used in the construction of explicit examples of manifolds with a certain geometric structure given by a suitable tensor. We refer to \cite{Tralle-Oprea} for more applications of these algebras in rational homotopy theory.

The classification of nilpotent Lie algebras of dimension $\leq 5$ does not present difficulties. The first classification in dimension 6 is apparently due to Umlauf, a student of Engel. Many modern classifications are available, see for instance \cite{B-M,Cicalo,deGraaf,Magnin,Morozov}. The approach of \cite{B-M} is the one we will pursue in this paper. In dimension seven, the problem becomes much harder. In \cite{Seeley} the complex case is tackled, and a full classification over $\RR$ was obtained in \cite{Gong}. A general classification (for an arbitrary field) is still lacking. In dimension 7 one can restrict to the smaller class of {\em 2-step} nilpotent Lie algebras, those whose commutator ideal is contained in the center. This was done in \cite{B} using similar techniques as the ones used here, and yields a classification over any field of characteristic not 2.

In this paper we focus on classifying 2-step nilpotent Lie algebras of dimension $8$ over an algebraically closed field $\kk$ of characteristic $\ne 2, 3$. This result generalizes other similar results in the literature. Over the field $\CC$ of complex numbers, the classification of the irreducible such Lie algebras, i.e. those which are not sums of lower-dimensional algebras, can be found in \cite{Ren2011,Yan}. More precisely, \cite{Ren2011} tackles irreducible Lie algebras with two-dimensional center, while \cite{Yan} deals with centers of dimensions three and four. On the other hand, degenerations of complex 2-step nilpotent Lie algebras have been studied in \cite{Alvarez}.

However, the techniques from \cite{Ren2011,Yan} are very different from the ones we use. They work directly in the Lie algebra with the bracket, and use the concept of {\em minimal systems of generators}, which is more specific of Lie theory. In order to distinguish the different algebras, they use some algebraic invariants (semi-simple derivations) that involve non-trivial machinery.
Instead, we work in the dual space of the Lie algebra, which is a {\em minimal differential graded algebra} (see Section \ref{sec:preliminaries} for the relevant definitions); the bracket is then dualized to a bivector. Moreover, the invariants we use to distinguish the algebras have an algebro-geometric interpretation as relative positions of a linear subspace with respect to certain algebraic varieties appearing in the stratification by the rank of the bivectors. With this approach, we obtain the following result:

\begin{theorem} \label{thm:main}
Let $\kk$ be any algebraically closed field of characteristic $\ne 2,3$. There are $35$ isomorphism classes of $8$-dimensional minimal algebras generated in degree $1$ over $\kk$, whose characteristic filtration has length 2.
\end{theorem}

This is a consequence of the analysis in Subsection \ref{easy:cases} and in Sections \ref{sec:(6,2)}, \ref{sec:(5,3)} and \ref{sec:(4,4)}; explicit models for such algebras are contained in Table \ref{table:final}. By the correspondence that assigns a differential on the exterior algebra of $V$ to a Lie algebra structure on the vector space $\frg=V^*$ (valid on any field $\kk$ of characteristic $\neq 2$, see Section \ref{sec:preliminaries}), we obtain the following:

\begin{corollary}
Let $\kk$ be any algebraically closed field of characteristic $\ne 2,3$. There are $35$ isomorphism classes of $8$-dimensional 2-step nilpotent Lie algebras over $\kk$.
\end{corollary}

As far as the authors know, the classification over arbitrary algebraically closed fields of characteristic $\ne 2, 3$ is a novel result, and it contains the previously mentioned classification available for the irreducible complex Lie algebras.
Apart from being valid over a more general field, the method we follow does not require to distinguish cases according to irreducibility, as both irreducible and reducible algebras appear naturally in the same line of thought. 
From the viewpoint of using the classification in other contexts (for instance, in order to construct nilmanifolds), it is very useful to collect all the 2-step nilpotent algebras in the same table, both reducible and irreducible. 

In addition, the classification presented here is self-contained, and uses mainly elementary and constructive methods. If one starts with any 8-dimensional 2-step nilpotent Lie algebra over $\kk$ with a given system of generators, one can locate it in the corresponding table and find the associated standard model, following these steps: first, dualize and obtain the structure equations of the corresponding minimal differential graded algebra; second, consider the linear subspace generated by the differentials of degree-1 elements, and compute its relative position with the strata by rank. With this information, it is possible to identify the algebra in the corresponding table. For each of the algebras, we provide a deduction of the standard model, i.e. a way of arriving at the model from the table, after several changes of bases.

In a forthcoming paper we shall use similar techniques to obtain the classification of $2$-step eight dimensional minimal algebras when the base field is $\RR$. These have been classified in the recent paper \cite{Borovoi2024} using more algebraically-flavoured methods. Recall that this type of nilpotent Lie algebras provides a rich source of examples for studying the behaviour of many types of geometric structures on nilmanifolds, for instance: complex structures and special Hermitian metrics (\cite{latorre2025}, \cite{latorre2017}), complex-symplectic structures (\cite{bazzoni2018}), and Spin(7)-structures (\cite{bazzoni2022}).
In this context, we hope our approach to the classification of 2-step algebras can provide a useful alternative.

\textbf{Acknowledgements.} We are grateful to Laurent Manivel, Vicente Muñoz and Giorgio Ottaviani for their help and suggestions. Also, we thank Ivan Kaygorodov for pointing us very useful references of related results we were unaware of. Special thanks to Riccardo Re for his help, and for engaging with us in several mathematical discussions. The first author is partially supported by the PRIN 2022 project ``Interactions between Geometric Structures and Function Theories''(code 2022MWPMAB) and by the GNSAGA of INdAM. The second author is partially supported by PID2021-126124NB-I00. This work was partially done during a research stay of the second author at the Dipartamento di Scienza ed Alta Tecnologia of Università dell'Insubria, funded by Universidad Politécnica de Madrid. The second author is grateful to the department for the welcoming environment.

\section{Preliminaries}\label{sec:preliminaries} 

\subsection{Minimal CDGA's}

A commutative differential graded algebra (CDGA, for short) over a field $\kk$ (of characteristic $\textrm{char}(\kk)\neq 2$) is a
graded $\kk$-algebra $A=\oplus_{k\geq 0} A^k$ such that $xy=(-1)^{|x||y|} yx$, for homogeneous
elements $x,y$, where $|x|$ denotes the degree of $x$, and endowed with a differential $d\colon A^k\to A^{k+1}$, $k\geq 0$, satisfying $d(xy)=(dx) y + (-1)^{|x|} x (dy)$, for homogeneous elements $x,y$. Morphisms between differential algebras are required to be degree-preserving algebra maps which commute with the differentials. Given a differential algebra $(A,d)$, we denote by $H^*(A)$ its cohomology. We say that a CDGA is \textit{connected} if $H^0(A)=\kk$.

A {\em minimal algebra}, is a CDGA $(A,d)$ of the following form:
\begin{enumerate}
 \item $A$ is the free commutative graded
 algebra $\Lambda V$ over a graded vector space $V=\oplus V^i$, 
 \item there exists a collection of generators $\{ x_\tau,
 \tau\in I\}$, for some well ordered index set $I$, such that
 $\deg(x_\mu)\leq \deg(x_\tau)$ if $\mu < \tau$ and each $d
 x_\tau$ is expressed in terms of preceding $x_\mu$ ($\mu<\tau$).
 This implies that $dx_\tau$ does not have a linear part.
\end{enumerate}

Minimal algebras are called nilpotent minimal algebras in \cite{Sullivan}. We have the following fundamental result: every connected CDGA $(A,d)$ has a \emph{minimal Sullivan model}; this means that there
exists a minimal algebra $(\Lambda V,d)$ together with a CDGA morphism
\[
\phi\colon(\Lambda V,d)\to(A,d)
\]
which induces an isomorphism on cohomology. The minimal model of a CDGA over a field $\kk$ of characteristic zero is unique up to isomorphism.


\medskip

Now we turn to the realm of Lie algebras.
To each Lie algebra $\frg$ we can associate the Chevalley-Eilenberg complex $(\Lambda\frg^*,d)$, whose differential is described according to the Lie algebra structure of $\frg$; namely, if $\{X_i\}$ is a basis for $\frg$ and $\{x_i\}$ denotes the dual basis for $\frg^*$, then
\begin{equation}\label{eq:224}
 d x_k(X_i,X_j)=-x_k([X_i,X_j]).
\end{equation}

Now suppose that $\frg$ is a \textit{nilpotent} Lie algebra; then there exists an ordered basis $\{X_i\}$ of $\frg$ such that
 \begin{equation}\label{eq:star}
 [X_i,X_j]=\sum_{k>i,j}a_{ij}^kX_k\, .
 \end{equation}
where $\{a_{ij}^k\}$ are called \emph{structure constants} of $\frg$. Therefore the differential can be written as
\begin{equation}\label{eq:224-2}
 dx_k= - \sum_{i,j<k}a_{ij}^k x_ix_j \,;
\end{equation}
where we write $x_ix_j\coloneqq x_i\wedge x_j$. This means that the Chevalley-Eilenberg complex associated to a nilpotent Lie algebra is a
\emph{minimal algebra generated in degree $1$}. Therefore, to study minimal algebras
generated in degree $1$ is equivalent to study nilpotent Lie algebras.

We want to rephrase the 2-step nilpotency condition on a nilpotent Lie algebra in the language of minimal algebras generated in degree 1. Let $(\L V, d)$ be a minimal CDGA over a field $\kk$. Its {\em characteristic filtration} $W_0 \subset W_1 \subset \ldots \subset V$ is defined as 
\[
W_0\coloneqq \ker(d)\,, \quad W_k\coloneqq d^{-1}(\L^2 W_{k-1}) \ \tr{for } k \ge 1\,. 
\]

The minimality condition implies that $W_k=V$ for some $k$. The {\em length} of the characteristic filtration is the minimal integer $n$ such that $W_n=V$. By \cite[Lemma 3]{B}, to study $n$-step nilpotent Lie algebras is equivalent to study minimal algebras generated in degree 1 whose characteristic filtration has length $n$.

Therefore, in this paper we classify 2-step nilpotent Lie algebras in dimension 8 by classifying 8-dimensional minimal algebras generated in degree $1$ whose characteristic filtration has length 2. We give a complete and explicit list of all such minimal algebras defined over $\kk$, producing one explicit representative of each isomorphism class.

\subsection{Characteristic filtration}

Let $(\L V, d)$ be an eight dimensional minimal CDGA over a field $\kk$ whose characteristic filtration has length $2$, i.e. $W_1=V$. Let $F_0=W_0$, and $F_1=W_1/W_0$. We can view $F_1$ as a subspace of $V$ by selecting (non-canonically) a subspace $F_1 \subset V$ such that $V=W_1=W_0 \oplus F_1$.

Consider the differential restricted to $F_1 \subset V$, so in particular $d\colon F_1 \to \L^2 W_0$. Although the space $F_1$ is chosen non-canonically, its image under the differential, $\tr{Im}(d)$ is canonically determined, in particular, independent of the choice of $F_1$. Recall that $W_0=\ker(d)$, hence
\begin{equation} \label{eq:d-injective}
    d\colon F_1 \to \L^2 W_0
\end{equation}
is injective. In particular, the dimension of $F_1$ cannot be greater that the dimension of $\L^2 W_0$. Let us denote $f_i=\dim F_i$. We distinguish cases according to the numbers $f_0, f_1$.
We may denote the different cases by the pair $(f_0,f_1)$. The above properties yield $f_0+f_1=\dim V=8$, and $f_1 \le {f_0 \choose 2}=\dim \L^2 W_0$. This only allows only the following cases:
\[
(f_0,f_1)\in\{(8,0),(7,1),(6,2),(5,3),(4,4)\}\,.
\]

\subsection{Rank of a bivector}
Let $V$ a vector space with $\dim V=n$.
Given $\f \in \L^2 V$, we can view $\f$ as a bilinear form $\f: V^* \times V^* \to \kk$, or equivalently as a linear map $\f^\#: V^* \to V$. The rank of $\f$ is defined as the rank of $\f$ as a bilinear form, or equivalently the rank of $\f^\#$ as a linear map (the dimension of its image). Recall the following elementary result, which gives the canonical form of a skew-symmetric bilinear form.

\begin{lemma}
Let $V$ be a vector space of dimension $n$.
Any $\f \in \L^2 V$ has even rank $2r \le n$. Moreover, $\rank \f=2r$ if and only if there exist linearly independent vectors $x_1,y_1, \ldots , x_k, y_k$ such that $\f=x_1  y_1+ \dots + x_r  y_r$.
\end{lemma}
\begin{remark} \label{rem:action-GL(W)-on-bivectors}
The above shows that for any pair of bivectors $\f, \p \in \L^2 V$ of the same rank, there exists a linear automorfism $f \in \GL(V)$ so that $\rho(f) (\f)=\p$ via the canonical representation 
\[
\rho\colon \GL(V) \to \GL(\L^2 V) \, .
\] 
\end{remark}

\noindent Recall that if the we view bivectors as skew-symmetric bilinear maps on $V^*$ and represent these bivectors as anti-symmetric matrices (via the choice of a basis in $V$), then a matrix $P \in \GL(n,\kk)$ acts on skew-symmetric matrices $A \in \mathfrak{o}(n)$ is $\rho(P) A=P^t A P$, with $P^t$ the transpose matrix.

\noindent Note that the image of $\f^\#$ is precisely the linear space in $V$ generated by $x_1,y_1,\dots, x_k,y_k$. 
Hence to every $\f$ of rank $2r$ we can canonically associate a subspace of $V$ of dimension $2r$. This fact will be important in the sequel.
\begin{definition}
We define the subspace associated to a bivector $\f \in \L^2 V$ as
\[
U_{\f}= \{ \f(u, \cdot): u \in V^*\} \subset V
\]
where we interpret $V=(V^*)^*$. To lighten notation, if we have more bivectors $\f_1,\f_2,\ldots$ we will denote by $U_1,U_2,\ldots$ the associated subspaces.
\end{definition}

\noindent Suppose $\f=\sum_{i<j} a_{ij} x_i x_j$ in some basis $x_i$ of $V$, and denote $u_i \in V^*$ the dual basis. Then $U_{\f}$ is generated by $\f(u_k,\cdot)= -\sum_{i<k} a_{ik} x_i + \sum_{k<j} a_{kj} x_j$, with $1 \le k \le n=\dim V$.
Recall that $\dim U_{\f}= \mathrm{rank}(\f)$.

\subsection{Classification}

Since $V=F_0 \oplus F_1$, there is an adapted basis $x_1, \dots x_{8}$ of $V$ so that $F_0=\ker d=\la x_1, \dots x_{f_0} \ra$ and $F_1=\la x_{f_0+1} \dots x_8 \ra$, and for any $f_0+1 \le f_0+k \le 8$ we have 
\[
\f_k=d x_{f_0+k} =\sum_{1 \le i<j\le f_0} a^k_{ij} x_i x_j \in \L^2 F_0 \, , \text{ for } 1 \le k \le f_1
\]
for some constants $a^k_{ij}$, the \emph{structure constants}.
Two minimal algebras $(\L V,d)$ and $(\L V',d')$ are isomorphic (by definition) if there is an isomorphism $V \cong V'$ that commutes with the differentials. Equivalently, $(\L V,d)$ and $(\L V',d')$ are isomorphic if we can find adapted bases $\{x_i\}$ of $V$, $\{x'_i\}$ of $V'$ (as above) with the same structure constants, so that the map sending $x_i$ to $x'_i$ is an isomorphism. 

\noindent Classifying these algebras consists in finding, for each isomorphism class, a basis $\{x_1,\dots,x_8\}$ of $V$ with the structural constants $a^k_{ij}$ as simple as possible. This is equivalent to finding a representative for the subspace $\tr{Im}(d) \subset \L^2 W_0$ under the action of $\GL(W_0)$ on $\L^2 W_0$. Recall that the differential in \eqref{eq:d-injective} is injective, hence the choice of a basis $\{u_1,\ldots,u_{f_1}\}$ of $F_1$ gives a basis $\{\f_1,\ldots \f_{f_1}\}$ of $\tr{Im}(d)$. In the presence of a basis $\{x_1,\ldots,x_{f_0}\}$ of $F_0=W_0$, two actions come into play:
\begin{itemize}
    \item the action of $\GL(F_1)$ on $\tr{Im}(d)$, changing $\{\f_k\}$ to $ \{\f'_k\}$.
    \medskip
    \item the action of $\GL(W_0)$ on $\L^2 W_0$, induced by a change of basis $\{x_i\} \mapsto \{x'_i\}$ in $W_0$.
\end{itemize}
In order to obtain a classification, we must find a basis of $W_0$ and a basis of $F_1$ so that $\tr{Im}(d)$ admits generators $\{\f_k\}$ as simple as possible. More precisely, we must find a suitable representative of the orbit $\GL(W_0) \cdot \tr{Im}(d) \subset \mathrm{Gr}(\L^2 W_0, f_1)$ in the Grassmannian of subspaces of $\L^2 W_0$ of dimension $f_1$, and suitable generators $\{\f_1,\ldots,\f_{f_1}\}$ whose expressions are simple in the sense that they involve the least number of sums and products, depending on its rank. Special changes of bases are given by homothethies, hence we will often work in the projectivization of these spaces. 

\subsection{Easy cases}\label{easy:cases}
Unless otherwise stated, from now on we suppose that the base field $\kk$ is \emph{algebraically closed and has characteristic $\ne 2, 3$}. 
Let us briefly comment the easiest cases, which are the algebras with $(f_0,f_1)$ equal to $(8,0)$ and $(7,1)$.

\noindent \textbf{Case (8,0)} In this case $W_0=V$ so the differential is identically zero and we have the trivial minimal algebra with $d=0$.

\noindent \textbf{Case (7,1)} In this case $f_0=7$, $f_1=1$, $d\colon F_1 \to \L^2 W_0$ is injective, so $\tr{Im}(d)=\la\f\ra$. We have three cases according to the rank of $\f$.
\begin{enumerate}
    \item If $\rank \f=2$ then $\f= x_1x_2 $ for some $x_1, x_2 \in W_0$ linearly independent. We complete to a basis $\{x_1,x_2,x_3,\dots, x_7\}$ of $W_0$, and we select $x_8 \in F_1$ so that $dx_8=\f=x_1x_2$. Hence we have a basis $\{x_1,\ldots,x_8\}$ of $V$ with $dx_8=x_1x_2$ and the rest of differentials zero.
    \item If $\rank \f=4$ then $\f= x_1x_2+x_3x_4$ for some linearly independent vectors $x_i \in W_0$, $i=1,\ldots,4$. As above, we complete it to a basis $\{x_1,\ldots, x_8\}$ of $V$ with $dx_8=x_1x_2+x_3x_4$ and $d\equiv 0$ on the remaining generators.
    \item If $\rank \f=6$ then by an analogous reasoning we get a basis $\{x_1,\ldots, x_8\}$ of $V$ with $dx_8=x_1x_2+x_3x_4+x_5x_6$ and $d\equiv 0$ on the remaining generators.
\end{enumerate}

\section{Case (6,2)}\label{sec:(6,2)}

\noindent We have that $f_0=6, f_1=2$; by \eqref{eq:d-injective} the differential determines a $2$-dimensional subspace $\tr{Im}(d)\subset \L^2 W_0$.
Note that $\dim W_0=6$, so $\dim \L^2 W_0=15$, and the rank of a non-zero bivector can be $2$, $4$ or $6$. Consider the projectivization $\PP^{14}=\PP (\L^2 W_0)$. The subspace $\tr{Im}(d)$ gives a line $\ell=\PP(\tr{Im}(d)) \subset \PP^{14}$. We need to study the possible ranks of the points of $\ell$, that is, the possible relative positions of $\ell$ with respect to the varieties of $\PP(\L^2 W_0)$ given by the bivectors of rank $2$, $4$, and $6$. 

\subsection{Stratification of $\L^2 W_0$}

Denote $W_0=W$. We study the stratification by rank in $\L^2 W$, for $W$ a vector space over a field $\kk$.
Note that a bivector $\f$ has rank $\le 2$ if and only if $\f^2=0$, and it has rank $\le 4$ if and only if $\f^3=0$. The condition of having rank $6$ is given by $\f^3 \ne 0$, hence it is open. 
The set of rank-$2$ and rank-$4$ bivectors are affine algebraic varieties of $\L^2 W$. Indeed, fix any basis $\{x_1,\dots,x_6\}$ of $W$ and consider the basis $\{x_ix_j \mid i \ne j \}$ of $\L^2 W$. 
A bivector $\f=\sum a^{ij}x_ix_j$ satisfies $\f^2=0$ iff all the coefficients of products of type $x_ix_jx_kx_l$ for $i<j<k<l$ vanish, and this is equivalent to the vanishing of $\binom{6}{4}=15$ equations of the form 
\[
a^{ij}a^{kl}-a^{ik}a^{jl}+a^{il}a^{jk}=0 \, , 1\le i<j<k<l \le 6 .
\]
Alternatively, this set can be seen as the Grassmannian of projective lines in $\PP(W)$, via the Plücker embedding.
On the other hand, the condition $\f^3=0$ is equivalent to the vanishing of a single cubic equation,
\begin{align}\label{eq:cubic}
    0=a_{12}a_{34}a_{56} - a_{13}a_{24}a_{56} + a_{14}a_{23}a_{56} - a_{15}a_{23}a_{46} + a_{16}a_{23}a_{45} + \nonumber \\
-a_{12}a_{35}a_{46} + a_{13}a_{25}a_{46} - a_{14}a_{25}a_{36} + a_{15}a_{24}a_{36} - a_{16}a_{24}a_{35} + \nonumber\\
+a_{12}a_{36}a_{45} - a_{13}a_{26}a_{45} + a_{14}a_{26}a_{35} - a_{15}a_{26}a_{34} + a_{16}a_{25}a_{34}\, .
\end{align} 
In the projectivization $\PP(\L^2W)=\PP^{14}$ we have the stratification by rank with strata:
\begin{align*}
    \GG& = \{ \f \in \L^2 W \mid \f^2=0 \} \, ; \\
    \cC & =\{ \f \in \L^2 W \mid \f^3=0 \}\, .
\end{align*}
with $\cC$ a cubic projective hypersurface in $\PP^{14}$, known as the {\em Pfaffian hypersurface}, and $\GG$ can be identified with the set of vector $2$-planes of $W$, i.e. the Grassmannian $\mathrm{Gr}(2,6)$, or equivalently the Grassmannian $\GG(1,5)$ of projective lines of $\PP^5=\PP(W)$.

\noindent Clearly, the group $\PGL(W)$ acting on $\PP(\L^2W)=\PP^{14}$ preserves the stratification given by rank. In particular, it preserves both the cubic $\cC$ and the Grassmannian $\GG$. There are three orbits for the action of $\PGL(W)$ on $\PP(\L^2 W)$. If we select a basis $\{x_1,\dots,x_6\}$ of $W$, these are: 
\begin{itemize}
\item the orbit of $[x_1x_2]$: the Grassmannian $\GG(1,5)=\GG$. 
\item the orbit of $[x_1x_2+x_3x_4]$: the set $\cC \setminus \GG$ of rank-$4$ bivectors.
\item the orbit of $[x_1x_2+x_3x_4+x_5x_6]$: the set $\PP^{14} \setminus \cC$ of rank-$6$ bivectors.
\end{itemize}

\begin{proposition} \label{prop:C-is-singular}
With notations as above, $\GG(1,5) \subset \PP^{14}$ is a smooth variety of dimension $8$ and degree $14$, and $\cC$ is an irreducible cubic hypersurface of $\PP^{14}$ whose set of singular points is $\GG(1,5)$.
\end{proposition}
\begin{proof}
The smoothness, dimension and degree of $\GG(1,5)$ follow by the usual properties of the Plücker embedding, see for instance \cite[p. 245]{Harris}. Let us see the claims about the cubic. Recall the action $\rho \colon \GL(W) \to \GL(\L^2 W)$, which preserves $\cC$ so for any $f \in \GL(W)$ we have $\rho(f)|_\cC$ an automorphism of $\cC$. This shows that the action of the automorphism group of $\cC$ has two orbits: a dense orbit given by $\cC \setminus \GG$, and $\GG$. It follows that $\cC$ is irreducible by an easy case analysis:
\begin{itemize}
    \item if $\cC=Q \cup H$ is a smooth quadric and a hyperplane, then $\Aut(\cC)$ would preserve $Q$ and $H$, impossible.
    \item if $\cC$ decomposes as $H_1 \cup H_2 \cup H_3$ with $H_i$ hyperplanes, then $\Aut(\cC)$ would preserve the intersection $H_1 \cap H_2 \cap H_3$.
    \item if $\cC=H_1^2 \cup H_2$, $\Aut(\cC)$ would preserve the intersection $H_1 \cap H_2$.
    \item lastly, $\cC$ cannot be $H^3$, because $\cC$ is not an hyperplane set-theoretically, for instance $x_1x_2+x_3x_4$ and $x_5x_6$ are in $\cC$, but its linear combinations are not.
\end{itemize}
This proves that $\cC$ is irreducible, so its singular points form a subvariety of positive codimension. No point $\f \in \cC \setminus \GG$ can be singular, as $\Aut(\cC)$ is transitive in this dense open set. Since $\Aut(\cC)$ acts transitively also in $\GG \subset \cC$, we are finished by showing that any particular point of $\GG$ is singular. For instance, we take the point $x_1x_2$ with coordinates $a_{12}=1, a_{ij}=0$ and we easily obtain that all partial derivatives of the equation \eqref{eq:cubic} vanish at this point.
\end{proof}

\subsection{Relative positions.}

We aim to study the relative positions of a line $\ell$, the rank-$2$ stratum $\GG=\GG(1,5)$, and $\cC$ inside $\PP^{14}=\PP(\L^2 W)$. 
For each relative position, we shall give a model for the corresponding minimal algebra by selecting generators of $\ell$ as simple as possible when expressed with respecto to a suitable basis.

\noindent Recall that if $\kk$ is algebraically closed, then either $\ell \subset \cC$ or $\ell \cap \cC$ consists of three points counted with multiplicity by Bezout's theorem. Obviously, a general line of $\PP^{14}$ is not contained in $\cC$. However, $\cC$ contains many lines: for instance any line generated by two rank-$2$ bivectors.

\noindent In fact, consider any $4$-dimensional subspace $Y \subset W$, so $\PP^5 \cong\PP(\L^2 Y) \subset \PP(\L^2 W)=\PP^{14}$. We can embed $\GG(1,3)$ in $\PP(\L^2 Y) \cong \PP^5$ via the Plücker embedding, and we have $\GG(1,3)=\GG\cap\PP(\L^2 Y)$. The image of the Plücker embedding is the so-called \emph{Klein quadric}, which is a non-degenerate quadric in $\PP^5$, ruled by planes, i.e. it contains a pair of transversal $2$-planes at each of its points. In particular, the Klein quadric $\GG(1,3)$ contains many lines, and $\GG(1,3)\subset\GG\subset \cC$.

\begin{proposition}\label{prop:case62-rel-positions-l-subset-C}
Notations as above. Suppose that the line $\ell$ is contained in $\cC$. Then one and only one of the following occurs:
\begin{enumerate}
\item $\ell \subset \GG$. 
In this case $\ell$ is generated by $\f_1=x_1x_2$, $\f_2=x_1x_3$ in suitable coordinates of $V$.

\medskip

\item $\ell \cap \GG=\{p\}$ and $\ell$ is contained in some $\PP^5=\PP(\L^2 Y)$, for some $Y\subset W$ with $\dim Y=4$. In this case $\ell$ is generated by $\f_1=x_1x_2$, $\f_2=x_1x_3+x_2x_4$, in suitable coordinates.

\medskip

\item $\ell \cap \GG=\{p\}$ and $\ell$ is not contained in some $\PP^5=\PP(\L^2 Y)$, for any $Y\subset W$ with $\dim Y=4$. In this case $\ell$ is generated by $\f_1=x_1x_2$, $\f_2=x_1x_3+x_4x_5$, in suitable coordinates.

\medskip

\item $\ell \cap \GG=\{p,q\}$. In this case $\ell$ is generated by $\f_1=x_1x_2$ and $\f_2=x_3x_4$, in suitable coordinates.

\medskip

\item $\ell \cap \GG=\emptyset$, and $\ell$ is contained in some $\PP^9=\PP(\L^2 U)$, for some $U\subset W$ with $\dim U=5$. In this case $\ell$ is generated by $\f_1=x_1x_2+x_3x_4$ and $\f_2=x_1x_4+x_3x_5$ in suitable coordinates.

\medskip

\item $\ell \cap \GG=\emptyset$ and $\ell$ is not contained in some $\PP^9=\PP(\L^2 U)$, for any $U\subset W$ with $\dim U=5$. In this case $\ell$ is generated by $\f_1=x_1x_2+x_3x_4$, $\f_2=x_1x_5+x_3x_6$, in suitable coordinates.

\end{enumerate}
Each of the above relative positions determines the orbit of $\ell$ under the action of $\GL(W)$, as we have a standard model for each of them in suitable coordinates of $W$.
\end{proposition}
\begin{proof}
An easy calculation using the homogeneous coordinates $[a_{12}:\dots: a_{56}]$ of $\PP(\L^2 W)=\PP^{14}$ shows that the six proposed models for a line $\ell \subset \cC$ satisfy the corresponding relative positions of items $(1)$-$(6)$. We need to show the opposite, namely that, under the condition $\ell \subset \cC$, the six relative positions of $\ell$ with $\GG$ and with the linear spaces $\PP^5$ and $\PP^9$ listed above exhaust all the possibilities, and that each one determines a unique model, i.e. a unique orbit for $\ell$ under the $\PGL(W)$ action.

\noindent We shall do this by simplifying an initial model, via changes of basis of $V$ and changes of generators of $\ell$. To avoid cumbersome notation we denote by $\f_1, \f_2$ some generators for $\ell$, that may change along the process, and by $\{x_1,\ldots,x_6\}$ a basis for $W$ that may also change. We analyze the different cases, which are collected in Table \ref{table:6-2C_1}.

\item \textbf{Case 1.} Suppose that $\ell$ intersects $\GG$ in at least two points $p_1=[\f_1]$, $p_2=[\f_2]$. Denote $U_1$, $U_2$ the associated $2$-planes.

\item \textbf{Subcase 1.1.} If $U_1$, $U_2$ intersect in a line, then we can choose a suitable basis so that
\[
\left\{\begin{array}{ccl}
 \f_1 & = & x_1x_2\\
 \f_2 & = & x_1x_3\\
\end{array}\right.
\]
All linear combinations $a\f_1+b\f_2$ have rank $2$, so in fact $\ell \subset \GG$. 

\item \textbf{Subcase 1.2.} if $U_1 \cap U_2=\{0\}$, then we can choose a suitable basis so that
\[
\left\{\begin{array}{ccl}
 \f_1 & = & x_1x_2\\
 \f_2 & = & x_3x_4\\
\end{array}\right.
\]
All linear combinations $a\f_1+b\f_2$ with $ab\neq 0$ have rank $4$, and $\ell \cap \GG=\{p_1,p_2\}$.

\item \textbf{Case 2.} Assume that $\ell \cap \GG$ is a point $p_1=[\f_1]$ with associated $2$-plane $U_1$. Select another point $p_2=[\f_2] \in \ell$ of rank $4$, with associated $4$-plane $U_2$. 
If $U_1 \cap U_2=\{0\}$, then in a suitable basis we would have $\f_1=x_5 x_6$, $\f_2=x_1 x_2 + x_3 x_4$. But in this case the line $\ell$ would contain bivectors of rank $6$, and this cannot happen, since $\ell\subset\cC$. We have the following subcases.

\item \textbf{Subcase 2.1.} Assume $U_1 \subset U_2$, so $\ell \subset \PP(\L^2 U_2)=\PP^5$. Choose a basis $\{x_1,x_2\}$ of $U_1$ so that $\f_1=x_1x_2$ and complete it to a basis $\{x_1,x_2,x_3,x_4\}$ of $U_2$. Then 
\[
\f_2=x_1(ax_2 + bx_3+cx_4) + x_2 (ex_3+fx_4) + gx_3x_4
\]
for some constants $a,b,c,e,f,g$. By setting $\f'_2=\f_2-a\f_1$ we can achieve $a=0$. Either $b$ or $c$ are non-zero, otherwise $U_2$ would have dimension less than $4$; hence, we can assume $b \ne 0$ by permuting $x_3,x_4$ if necessary. Rescale $x_3$ so that $b=1$ (i.e. change $x_3$ by $bx_3$), and make the change of basis $x'_3=x_3+cx_4$ so that $\f_2=x_1x'_3 + x_2 (ex'_3+fx_4) + gx'_3x_4$, for differents constants $e,f,g$. Note that $f\ne 0$, so by an analogous procedure with $x_4$ we can assume $f=1$ and consider $x'_4=x_4+ex'_3$. Reset notation $x'_3=x_3$, $x'_4=x_4$, and we get $\f_2=x_1x_3 + x_2 x_4 + gx_3x_4$. 
If it were $g \ne 0$ then we could assume $g=1$ by rescaling $x_3, x_1$ and $\f_1$, and then $\f_1+\f_2=x_1(x_2+x_3)+(x_2+x_3)x_4=(x_1-x_4)(x_2+x_3)$, a contradiction. Hence $g=0$ and we get to the model 
\[
\left\{\begin{array}{ccl}
 \f_1 & = & x_1x_2\\
 \f_2 & = & x_1x_3+x_2x_4\\
\end{array}\right.
\]

\item \textbf{Subcase 2.2.} If $\dim(U_1 \cap U_2)=1$ then $U\coloneqq U_1 + U_2$ has dimension $5$ and $\ell\subset \PP(\L^2U)=\PP^9$. Choose an initial basis with $\f_1=x_1x_2$, $U_2=\la x_1, x_3, x_4, x_5 \ra$, and
\[
\f_2=x_1(ax_3 + bx_4+cx_5) + x_3 (ex_4+fx_5) + gx_4x_5
\]
with one of $a,b$ or $c$ non-zero. By permuting $x_3$ with $x_4$ or $x_5$ we can assume $a \ne 0$, and rescale so that $a=1$. Consider $x'_3=x_3 + bx_4+cx_5$ and reset notation, so $\f_2=x_1x_3+x_3(ex_4+fx_5)+gx_4x_5$, with $g \ne 0$. Rescale to get $g=1$ and note that $\f_2=x_1x_3 + x_4(x_5+ex_3)+fx_3x_5$, so put $x'_5=x_5+ex_3$, reset notation, and $\f_2=x_1x_3+(x_4+fx_3)x_5$. Making a last change $x'_4=x_4+fx_3$ yields the model 
\[
\left\{\begin{array}{ccl}
 \f_1 & = & x_1x_2\\
 \f_2 & = & x_1x_3+x_4x_5\\
\end{array}\right.
\]

\item \textbf{Case 3.} Suppose that $\ell \subset \cC$ but $\ell\cap\GG=\emptyset$. Take two points $p_1=[\f_1]$, $p_2=[\f_2]$ in $\ell$, both with rank $4$. Notice that we cannot have $U_1=U_2$. Indeed, if this was the case, then $\ell \subset \PP^2(\L^2 Y)=\PP^5$ with $Y=U_1=U_2$. The rank-$2$ bivectors of this $\PP^5$ form the Klein quadric $\GG(1,3) \subset \PP^5$, and by Bezout's theorem $\ell$ intersects $\GG(1,3)$, in particular $\ell$ would contain points of rank $2$, which is absurd. The following subcases arise.

\item \textbf{Subcase 3.1.} If $\dim(U_1 \cap U_2)=3$, then $U\coloneqq U_1 + U_2$ has dimension $5$, so the line $\ell$ is contained in $\PP(\L^2 U)=\PP^9$. Take an initial basis so that $U_1=\la x_1,x_2,x_3,x_4 \ra$ and $\f_1=x_1x_2+x_3x_4$. We can assume (permuting the basis elements) that $x_4 \notin U_2$. The affine lines $x_i + \la x_4 \ra$, $i=1,2,3$, intersect $U_2$, since $U_2$ is a hyperplane of $U$. We can make a change of basis $x'_i=x_i+a_i x_4$ so that $x'_i \in U_2$. In the new basis we have $\f_1=x'_1x'_2+(x'_3-a_2x'_1+a_1x'_2)x_4$. With the further change of basis $x''_3=x'_3-bx'_1+ax'_2$, and resetting notation, we get $\f_1=x_1x_2+x_3x_4$ and $U_2=\la x_1,x_2,x_3, x_5\ra$. 
We have then
\[
\f_2=x_1(ax_2+bx_3+cx_5)+x_2(ex_3+fx_5)+gx_3x_5\,.
\]
Suppose $a=0$; since $\f_2$ has rank 4, we must have $ce-bf\ne 0$. As one of $c$ or $f$ is non-zero, we can assume $f\ne 0$ (permuting $x_1,x_2$ and changing the sign of $x_3, x_4$ if needed, in order to keep the expression of $\f_1$ fixed). Rescaling $x_5$ we get $f=1$. Consider the change $x'_5=ex_3+x_5$, and reset notation, so that, in the new basis, $\f_2=x_1(bx_3+cx_5)+x_2x_5+gx_3x_5$, with $b \ne 0$. Rescale $x_3, x_4$ 
so that $b=1$, and put $\f_2=x_1x_3+(cx_1+x_2+gx_3)x_5$. Set $x'_2=cx_1+x_2+gx_3$, so that $\f_2=x_1x_3+x'_2x_5$ and $\f_1=x_1x'_2+x_3(gx_1+x_4)$. With $x'_4=gx_1+x_4$ we get the model
\[
\left\{\begin{array}{ccl}
 \f_1 & = & x_1x_2+x_3x_4\\
 \f_2 & = & x_1x_3+x_2x_5\\
\end{array}\right.
\]
If $a\neq 0$ we can assume $a=1$ by rescaling $\f_2$, hence
\[
\f_2=x_1(x_2+bx_3+cx_5)+x_2(ex_3+fx_5)+gx_3x_5
\]
with $g-bf+ce\neq0$. If $g=0$ then $ce-bf\neq0$; arguing as above, we can assume $f\neq 0$, and rescale $x_5$ to achieve $f=1$. Do the change $x'_5=ex_3+x_5$ and reset notation so that, in the new basis, $\f_2=x_1(x_2+bx_3)+(cx_1+x_2)x_5$. The change $x_2'=x_2+cx_1$ gives (upon renaming) $\f_2=bx_1x_3+x_2(x_5-x_1)$, with $b \ne 0$ as the rank of $\f_2$ is 4. The final change $x_5'=x_5-x_1$ leaves us with $\f_2=bx_1x_3+x_2x_5$. By suitably rescaling $x_3$ and $x_4$ we obtain $b=1$, hence again the above model. Finally, we treat the case $g \ne 0$. First, rescale $x_5$ to get
\[
\f_2=x_1x_2+bx_1x_3+cx_1x_5+ex_2x_3+fx_2x_5+x_3x_5=(x_1-ex_3)(x_2+bx_3)+(cx_1+fx_2+x_3)x_5
\]
with $1-bf+ce\neq 0$. Setting $x'_1=x_1-ex_3$, $x'_2=x_2+bx_3$, $x'_4=x_4+bx'_1+ex'_2$, and renaming variables we get $\f_1=x_1x_2+x_3x_4$ and $\f_2=x_1x_2+(cx_1+fx_2+(1-bf+ce)x_3)x_5$. In fact, rescaling $x_5$ we achieve $\f_2=x_1x_2+(cx_1+fx_2+x_3)x_5$. If $c=f=0$ we reach a contradiction, since $\f_1-\f_2$ has rank 2, which implies $\ell\cap\GG\neq\emptyset$. Upon switching $x_1$ and $x_2$ we can assume $f\neq 0$ and rescale $x_5, x_3, x_4$ so that $f=1$. Put $x'_2=x_2+cx_1+x_3$ and $x'_4=x_4+x_1$ so that, after renaming, $\f_1$ stays the same and $\f_2=x_1(x_2-x_3)+x_2x_5=-x_1x_3+x_2(x_5-x_1)$, and now write $-x'_5=x_5-x_1$, so that $-\f_2$ yields our sought model.

\item \textbf{Subcase 3.2.} Suppose $\dim(U_1 \cap U_2)=2$. Pick $x_1 \in U_1 \cap U_2$ and complete it to a basis $\{x_1,x_2,x_3,x_4\}$ of $U_1$ such that $\f_1=x_1x_2+x_3x_4$. We claim that $\dim(U_2 \cap \la x_3, x_4 \ra)=1$. By contradiction, suppose that $U_2 \cap \la x_3, x_4 \ra =\{0\}$. The affine plane $x_2 + \la x_3, x_4 \ra$ must intersect $U_2$, so we find $a,b$ such that $x'_2=x_2+ax_3+bx_4\in U_2$. Then, $\f_1=x_1x'_2+(x_3-bx_1)(x_4+ax_1)=x_1x'_2+x'_3x'_4$ and, resetting notation, we get $\f_1=x_1x_2+x_3x_4$ and $U_2=\la x_1, x_2, x_5, x_6 \ra$. Then
\[
\f_2=ax_1x_2+bx_1x_5+cx_1x_6+ex_2x_5+fx_2x_6+gx_5x_6
\]
with $ag-bf+ce\neq 0$. Since $\ell\subset\cC$, $\f_1+\f_2$ must have rank 4, hence
\[
0=(\f_1+\f_2)^3=3(\f_1^2\wedge\f_2+\f_1\wedge\f_2^2)=6(ag-bf+ce+g)x_1x_2x_3x_4x_5x_6\,,
\]
and $g=-(ag-bf+ce)\neq 0$. Setting $x_5'=cx_1+fx_2+gx_5$ and $x'_6=x_6-bx_1-ex_2$ we get
\[
\f_2=(ag-bf+ce)x_1x_2+x'_5x'_6\,.
\]
But then some linear combinations of $\f_1$ and $\f_2$ would have rank 6, violating the condition $\ell\subset\cC$. Therefore we can assume $\dim(U_2 \cap \la x_3, x_4 \ra)=1$. In this case we can arrange that $U_2=\la x_1, x_3, x_5, x_6 \ra$, and $\f_2$ has the form
\[
\f_2=g x_5x_6+ x_1(ax_3+bx_5+cx_6)+x_3(ex_5+fx_6)\,.
\]
If $g\ne 0$, then we can assume $g=1$ and we get that $\f_2+\f_1$ has rank $6$, a contradiction. Hence $g=0$, and $b$ or $c$ must be non-zero, so we can assume $b \ne 0$, $b=1$ after rescaling $x_5$, and change $x'_5=ax_3+x_5+cx_6$, so $\f_2=x_1x_5+x_3(ex_5+fx_6)$. As $x_6 \in U_2$, it must be $f \ne 0$, so we get $\f_2=x_1x_5+x_3x'_6$ with $x'_6=fx_6+ex_5$. We arrive at the model:
\[
\left\{\begin{array}{ccl}
 \f_1 & = & x_1x_2+x_3x_4\\
 \f_2 & = & x_1x_5+x_3x_6
\end{array}\right.
\]
\end{proof}

\noindent We collect in Table \ref{table:6-2C_1} the results of Proposition \ref{prop:case62-rel-positions-l-subset-C}, in which the case $\ell\subset\cC$ is handled.
\begin{itemize}
    \item The second column contains the relative position of $\ell$ with respect to $\GG\subset\cC$;
    \item the third column contains the dimension $\delta$ of a subspace $\PP(\L^2 U) \subset \PP^{14}=\PP(\L^2 W)$, for $U\subset W$ a subspace, in which $\ell$ is contained. Clearly $\delta\in\{2,5,9,14\}$;
    \item the fourth and fifth columns contain the differentials of the non-closed elements;
    \item the sixth column says whether the minimal algebra is {\em irreducible}, i.e.~it is not the sum of lower-dimensional minimal algebras; notice that irreducibility is equivalent to $\delta=14$;
    \item in case it is irreducible, the seventh column identifies our algebra with the Lie algebra in the list obtained in \cite{Ren2011}.
\end{itemize}

\begin{table}[h]
\caption{Minimal algebras of type $(6,2)$ with $\ell\subset\cC$}\label{table:6-2C_1}
{\tabulinesep=1mm
\begin{tabu}{|c|c|c|c|c|c|c|c|c|c|c|c|c|c|}
\hline
Label & $\ell\cap\GG$ & $\delta$ & $dx_7$ & $dx_8$ & Irreducible & \cite{Ren2011}\\
\hline
(6.2.1) & $\ell$ & 2 & $x_1x_2$ & $x_1x_3$ & $\times$ &\\
\hline
(6.2.2)& $\{p_1,p_2\}$ & 5 & $x_1x_2$ & $x_3x_4$ & $\times$ & \\
\hline
(6.2.3) & $\{p\}$ & 5 & $x_1x_2$ & $x_1x_3+x_2x_4$ & $\times$ & \\
\hline
(6.2.4) & $\{p\}$ & 9 & $x_1x_2$ & $x_1x_3+x_4x_5$ & $\times$ & \\
\hline
(6.2.5) & $\emptyset$ & 9 & $x_1x_2+x_3x_4$ & $x_1x_3+x_2x_5$ & $\times$ & \\
\hline
(6.2.6) & $\emptyset$ & 14 & $x_1x_2+x_3x_4$ & $x_1x_5+x_3x_6$ & \checkmark & $N_2^{8,2}$\\
\hline
\end{tabu}}
\end{table}

\begin{remark}
The last two models in Table \ref{table:6-2C_1} have also been obtained in \cite[Proposition 2]{Manivel-Mezzetti}.
\end{remark}

\noindent Now assume that $\ell$ is not contained in $\cC$. Since $\cC$ is a cubic hypersurface of $\PP^{14}$, $\ell \cap \cC$ consists of three points counted with multiplicity, by Bézout's theorem. These points might be in $\GG$.
If we parameterize the line $\ell$ as $p+vt$ in an affine chart of $\PP^{14}$ around a point $p \in \ell \cap \cC$, substitute this parametrization in the equation of $\cC$, we get $O(t^k)$, for some $k=1,2,3$, and this exponent $k$ is the multiplicity of intersection, denoted $I_p(\ell,\cC)$. Recall that the points of $\GG$ are singular points of $\cC$, hence a line $\ell\subset\PP^{14}$ through $p\in \GG$ has $I_p(\ell,\cC)\ge 2$. In particular, if $\ell$ is not contained in $\cC$, then $\ell$ cannot pass through more than $1$ point of $\GG$. This can easily be seen by using coordinates, as it is clear that the span of two rank-$2$ bivectors contains bivectors of rank at most $4$.

\begin{proposition} \label{prop:case62-rel-positions-l-not-contained-C}
Notations as above. Suppose that the line $\ell$ is not contained in $\cC$. Then one and only one of the following occurs:
\begin{enumerate}
\item $\ell \cap \cC=\{p_1,p_2,p_3\}$, $\ell \cap \GG=\emptyset$. 
We can take generators for $\ell$ of the form $\f_1=x_1x_2+x_3x_4$, $\f_2=x_3x_4+x_5x_6$, with $p_i=[\f_i]$, and $\f_3=\f_1-\f_2$

\medskip

\item $\ell \cap \cC=\{p_1,p_2\}$, with multiplicities $2,1$ respectively, and $\ell \cap \GG=\emptyset$. 
In this case, $\ell$ is generated by two bivectors of the form $\f_1=x_1x_2+x_3x_4$, $\f_2=x_3x_5+x_4x_6$.

\medskip

\item $\ell \cap \cC=\{p\}$, with multiplicity $3$, and $\ell \cap \GG=\emptyset$. 
We can choose generators for $\ell$ of the form $\f_1=x_1x_2+x_3x_4$, $\f_2=x_1x_5+x_2x_3+x_4x_6$.

\medskip

\item $\ell \cap \cC=\{p_1,p_2\}$ with multiplicities $2,1$ respectively, and $\ell \cap \GG=\{p_1\}$. The generators are of type $\f_1=x_1x_2$, $\f_2=x_3x_4+x_5x_6$.

\medskip

\item $\ell \cap \cC=\{p\}$ with multiplicity $3$, and $\ell \cap \GG=\{p\}$. The generators are $\f_1=x_1x_2$, $\f_2=x_1x_3+x_2x_4+x_5x_6$.
\end{enumerate}
Each of the above relative positions determines a standard model for $\ell$ in suitable coordinates, hence the orbit of $\ell$ under the action of $\GL(W)$. 
\end{proposition}
\begin{proof}
Notice that the condition $\ell \not \subset \cC$ implies necessarily that for any generators $\f_1,\f_2$ of $\ell$ we have $U_1+U_2=W$; indeed, if it were $\dim(U_1+U_2)\leq 5$, we could find a basis of $W$ so that $\ell=\la \f_1, \f_2 \ra \subset \PP(\L^2 \la x_1, \dots , x_5 \ra) \subset \cC$.
Now we analyze the different cases; the results are contained in Table \ref{table:6-2C}.

\noindent \textbf{Case 1.} $\ell \cap \cC=\{p_1,p_2,p_3\}$, $\ell \cap \GG=\emptyset$. This is the generic case, meaning that a generic line satisfies this condition. If $\ell$ is generated by $\f_1=x_1x_2+x_3x_4$ and $\f_2=x_1x_2+x_5x_6$, then $\ell \cap \cC$ consists of the points $p_1=[\f_1]$, $p_2=[\f_2]$, $p_3=[\f_1-\f_2]$, and $\ell \cap \GG=\emptyset$; indeed, $t \f_1 +s \f_2=(t+s)x_1x_2+tx_3x_4+sx_5x_6$ has rank $6$ if $[t:s] \notin \{[1:-1],[1:0],[0:1]\}$, and rank $4$ otherwise.

\noindent On the opposite direction, let $\ell$ be a line with this relative position with $\cC$ and $\GG$, and let us see that there is a basis of $W$ so that $\ell$ has this model. As usual we start with coordinates so that $\f_1=x_1x_2+x_3x_4$. 

\noindent Now we need to discard the case in which $\dim(U_2 \cap \la x_1,x_2 \ra)=\dim(U_2\cap\la x_3,x_4\ra)=1$. If this were the case, we could assume that $U_2=\la x_1,x_3,x_5,x_6 \ra$ so $\f_2=x_1(ax_3+bx_5+cx_6)+x_3(ex_5+fx_6)+gx_5x_6$. If $b$ or $c$ are non-zero we can assume (switching $x_5$ and $x_6$ if needed, and rescaling) that $b=1$; if we put $x'_5=ax_3+x_5+cx_6$ we get $\f_2=x_1x_5+x_3(ex_5+fx_6)+gx_5x_6$, and now $f \ne 0$. Rescale so that $f=1$, and change $x'_6=ex_5+x_6$, so we get $\f_2=x_1x_5+x_3x_6+gx_5x_6$. If $g=0$ then $\ell \subset \cC$, a contradiction; and if $g \ne 0$ then it is easy to see that $\a \f_1 + \b \f_2$ has rank $6$ for any $\a \ne 0 \ne \b$, so $\ell \cap \cC$ is two points, a contradiction. If $b=c=0$ above, then $ag\neq 0$; we can achieve $g=1$, hence $\f_2=ax_1x_3+(x_5+fx_3)(x_6-ex_3)$; setting $x'_5=x_5+fx_3$ and $x'_6=x_6-ex_3$, rescaling and renaming we obtain $\f_2=x_1x_3+x_5x_6$, but this would give only two points in $\ell\cap\cC$.

\noindent Hence we can assume, permuting the pairs $(x_1,x_2)$ and $(x_3,x_4)$ if necessary, that $U_2 \cap \la x_1,x_2 \ra=\{0\}$, so we can take $U_2=\la x_5,x_6, x_3+ax_1, x_4+bx_2 \ra$. Make the change $x'_2=x_2-ax_4$ and $x'_3=x_3+ax_1$, so $\f_1=x_1x'_2+x'_3x_4$, and $U_2=\la x_5,x_6,x'_3, (1+ab)x_4+bx'_2 \ra$. Note that the case $ab+1=0$ contradicts the assumption $U_2 \cap \la x_1,x_2 \ra=\{0\}$, so we can divide by $1+ab$, and repeat this process with the change $x_1'=x_1-bx_3$, $x'_4=x_4+bx'_2$. With this, we get a model with $\f_1=x_1x_2+x_3x_4$ and $U_2=\la  x_3, x_4, x_5, x_6 \ra$. We can write then:
\[
\f_2=x_3(ax_4+bx_5+cx_6)+x_4(ex_5+fx_6)+gx_5x_6 \, .
\]
If $a=0$ then $bf-ce\neq 0$, so we find a basis of $U_2$ such that $\f_2=x_3x_5+x_4x_6+gx_5x_6$. An easy computation shows that $\a \f_1+ \b \f_2$ has rank $6$ unless $\a\b(\a g-\b)=0$. Since there must be three distinct points of rank $4$, $g \ne 0$, and we can assume $g=1$, so $\f_2=x_3x_5+(x_4+x_5)x_6$. We define $x'_5=-(x_4+x_5)$, $x'_6=x_3-x_6$ and we get $\f_2=-x_3x_4-x'_5x'_6$. After changing the sign of $\f_2$, we obtain the model
\[
\left\{\begin{array}{ccl}
 \f_1 & = & x_1x_2+x_3x_4\\
 \f_2 & = & x_3x_4+x_5x_6
\end{array}\right.
\]
On the other hand, if $a \ne 0$ we can rescale it to $a=1$. If, moreover, $g=0$ then we get easily to the expression $\f_2=x_3x_4+x_4x_5+x_3x_6$ and we get to a contradiction since $\a \f_1 + \b \f_2$ has rank $6$ whenever $\a, \b \ne 0$. Then it must be $g\ne 0$, so take $g=1$ and then $\f_2=x_3x_4+(x_5-cx_3-ex_4)x_6$ and with the change $x'_5=x_5-cx_3-ex_4$ we arrive at our model for $\ell$.

\medskip

\noindent \textbf{Case 2.} $\ell \cap \cC=\{p_1,p_2\}$ with multiplicities $2,1$, $\ell \cap \GG=\emptyset$. 
Hence $p_1,p_2$ are smooth points of $\cC$, $\ell$ intersects $\cC$ at $p_1$ with multiplicity $2$ and transversely at $p_2$. For instance, consider $\ell$ generated by $\f_1=x_1x_2+x_3x_4$ and $\f_2=x_3x_5+x_4x_6$. The line $\a \f_1 +\b  \f_2$, is given in coordinates by $a_{12}=a_{34}=\a$, $a_{35}=a_{46}=\b$; plugging this into \eqref{eq:cubic} we see that $\ell \cap \cC$ is given by $\a \b^2=0$, so $\f_1$ is indeed the double point of intersection. Now let us see that the above is the only model of a line $\ell$ intersecting $\cC$ in this way.

\item Take $\f_1$ as the double point of $\ell\cap\cC$ and choose initial coordinates so that $\f_1=x_1x_2+x_3x_4$. We claim that $\f_2$ satisfies that either $U_2\cap \la x_1, x_2 \ra=\{0\}$ or $U_2\cap \la x_3, x_4 \ra=\{0\}$. Indeed, if that was not the case, then we can assume that $x_1, x_3 \in U_2$, so $U_2=\la x_1,x_3,x_5,x_6 \ra$ and
\[
\f_2=x_1(ax_3+bx_5+cx_6)+x_3(ex_5+fx_6)+gx_5x_6 \, .
\]
If $e=f=0$ then $ag\neq 0$ and we can rescale to get $a=g=1$. We obtain $\f_2=x_1x_3+(x_5+cx_1)(x_6-bx_1)$ so after the obvious change we get $\f_2=x_1x_3+x_5x_6$. But then the intersection of the line $\ell=\la\a \f_1+\b \f_2\ra$ with $\cC$ is given by $\a^2\b=0$, so $\f_2$ is the double point of the intersection. This is a contradiction with our choice of $\f_1$ as the double point. We conclude that one of $e$ or $f$ is non-zero, so we can assume $e=1$ and put $x'_5=x_5+fx_6$, so that, upon renaming,
\[
\f_2=x_1(ax_3+bx_5+cx_6)+(x_3-gx_6)x_5
\]
We must have $g\ne 0$; indeed, if $g=0$ then we can assume $c=1$ and put $x'_6=ax_3+bx_5+x_6$, so $\f_2=x_1x_6+x_3x_5$ and $\ell \subset \cC$, a contradiction. So we can assume $g=1$ by rescaling $x_6$, and change $x'_6=x_6-x_3$ so $\f_2=x_1(ax_3+bx_5+cx_6)+x_5x_6$ with $a \ne 0$. By rescaling $x_3$, $x_2$, and $\f_1$ we get $a=1$ so $\f_2=x_1x_3+(x_5+cx_1)(x_6-bx_1)$ and with the change $x'_5=x_5+cx_1$ and $x'_6=x_6-bx_1$ we arrive at $\f_2=x_1x_3+x_5x_6$. As above, this displays a contradiction with $\f_1$ being the double point in $\ell\cap\cC$. This proves the claim about $U_2$. 

\item We can therefore assume that $U_2$ intersects trivially $\la x_1, x_2 \ra$ or $\la x_3, x_4 \ra$. By permuting the pairs $(x_1,x_2)$ and $(x_3,x_4)$ we can assume that $U_2 \cap \la x_1, x_2 \ra=\{0\}$. This yields $U_2=\la x_3+ax_1, x_4+bx_2, x_5,x_6 \ra$. Make the changes $x'_2=x_2-ax_4$, $x'_3=x_3+ax_1$, so that $\f_1=x_1x_2'+x'_3x_4$, and reset notation. Repeating the process with $x'_1=x_1+bx_3$, $x'_4=x_4+bx_2$ yields a basis with $\f_1=x_1x_2+x_3x_4$ and $U_2=\la x_3,x_4,x_5,x_6 \ra$. Write
\[
\f_2=x_3(ax_4+bx_5+cx_6)+x_4(ex_5+fx_6)+gx_5x_6 \, .
\]
If $b=c=0$, then $ag \ne 0$ so we can get $a=g=1$ and $\f_2=x_3x_4+x_4(ex_5+fx_6)+x_5x_6$. Now it is easy to see that after a suitable change we get $\f_2=x_3x_4+x_5x_6$: if $e=f=0$ this is clear, if $e\ne 0$ we can rescale to get $e=1$ and make the changes $x'_5=x_5+fx_6$, $x'_6=x_6-x_4$, so that $\f_2=x_3x_4+x'_5x'_6$. Since $\f_1=x_1x_2+x_3x_4$, we see that $\f_1-\f_2$ has rank $4$ so $\ell\cap\cC$ consists of three distinct points, a contradiction.
\item Hence, we can assume that one of $b$ or $c$ is non-zero, so after permuting $x_5,x_6$ if necessary we can rescale to have $b=1$ and change $x'_5=ax_4+x_5+cx_6$, so $\f_2=x_3x_5+x_4(ex_5+fx_6)+gx_5x_6$ with $f \ne 0$. We rescale so that $f=1$ and put $x'_6=ex_5+x_6$ so $\f_2=x_3x_5+(x_4+gx_5)x_6$. Now, if it was $g\ne 0$ then we could rescale $x_5$, $x_3$, $x_1$ and $\f_1$ in order to get $g=1$. If $\ell=\la\a \f_1+ \b \f_2\ra$, then $\ell \cap \cC=\{\f_1,\f_2,\f_1+\f_2\}$, which gives a contradiction. We are finally done: $g=0$ follows, and we get the model
\[
\left\{\begin{array}{ccl}
 \f_1 & = & x_1x_2+x_3x_4\\
 \f_2 & = & x_3x_5+x_4x_6
\end{array}\right.
\]

\medskip

\noindent \textbf{Case 3.} $\ell \cap \cC=\{p\}$ with multiplicity $3$ and $\ell \cap \GG=\emptyset$. This means that $p$ is a smooth point of $\cC$ and $\ell$ is tangent to $\cC$ at $p$ with multiplicity $3$.
Consider the bivectors $\f_1=x_1x_2+x_3x_4$ and $\f_2=x_1x_5+x_2x_3+x_4x_6$ and the projective line $\ell=\la\a \f_1+\b \f_2\ra\subset\PP^{14}$. Its parametric equations are $a_{12}=a_{34}=\a$, $a_{15}=a_{23}=a_{46}=\b$. Plugging them into \eqref{eq:cubic} we obtain $\b^3=0$, so $\f_1$ is indeed a triple point. Note that for any $\a$ we have $(\a \f_1 + \f_2)^3 \ne 0$.

\item Let us see that this is the only model satisfying this. Choose an initial basis so that $\f_1=x_1x_2+x_3x_4$ and $U_1=\la x_1,x_2,x_3,x_4\ra$; write the rank $6$ generator as $\f_2=\sum_{i<j} a_{ij} x_ix_j$. We claim first that some of the coefficients $a_{15}$, $a_{25}$, $a_{35}$, $a_{45}$ must be non-zero. Indeed, if they all vanished, then $a_{56} \ne 0$ and we could write 
\[
\f_2=(\textstyle{\sum_{i=1}^5 a_{i6} x_i})x_6 + \xi_2
\]
where $\xi_2 \in \L^2 U_1$ has rank $4$. The line generated by $\f_1, \xi_2$ in $\PP(\L^2 U_1)=\PP^5$ must contain some rank-$2$ bivector, since these form the Klein quadric, so there are non-zero scalars $\a_0, \b_0$ so that $\a_0 \f_1+ \b_0 \xi_2$ has rank $2$. But then $\a_0 \f_1+ \b_0 \f_2$ would have rank $4$, a contradiction. An analogous argument permuting $x_5, x_6$ shows that one of $a_{16}$, $a_{26}$, $a_{36}$, $a_{46}$ must be non-zero.

\item We can assume that $a_{15}$ or $a_{25}$ are non-zero permuting the pairs $(x_1, x_2)$ and $(x_3, x_4)$, and moreover that $a_{15} \ne 0$ changing $x_1$ by $x_2$ and $x_2$ by $-x_1$. We rescale so that $a_{15}=1$ and write
\[
\f_2=x_1( x_5 + \textstyle{\sum_{i\ne 5} a_{1i} x_i}) + \textstyle{\sum_{1<i<j\leq 6} a_{ij}x_ix_j}
\]
and make the change $x'_5=x_5 + \sum_{i\neq 5} a_{1i} x_i$ and reset notation so that 
\begin{equation}\label{eq:6.2.3}
\f_2=x_1x_5 + \textstyle{\sum_{1<i<j\leq 6} a_{ij}x_ix_j}\,.    
\end{equation}
As we noted above, some of $a_{26}$, $a_{36}$, $a_{46}$ must be non-zero, but more is true in this setting: it must be $a_{36} \ne 0$ or $a_{46} \ne 0$. Indeed, if $a_{36}=a_{46}=0$ then $a_{26} \ne 0$, so we can rescale $\f_2$ to achieve $a_{26}=1$. Make a change $x'_6=x_6+\sum_i a_{2i} x_i$, and reset notation so that the only monomial containing $x_2$ is $x_2x_6$, and note that $a_{36}=a_{46}=0$ in the new basis, so we have
\[
\f_2=x_1x_5 + x_2 x_6 + x_3(x_4+a_{35}x_5) + a_{45}x_4x_5 + a_{56}x_5x_6
\]
where we have put $a_{34}=1$ because $a_{34} \ne 0$ (otherwise $\f_2^3=0$) and we can rescale so that it equals $1$. 
But now we get to a contradiction since we can cancel the term $x_3x_4$ by considering
\[
\f_2-\f_1=(x_1+a_{35}x_3+a_{45}x_4-a_{56}x_6)x_5 + x_2(x_6+x_1)
\]
and this has rank $4$. We conclude that either $a_{36}$ or $a_{46}$ are non-zero in \eqref{eq:6.2.3}, so permuting $x_3, x_4$ we can assume that $a_{46} \ne 0$, rescale so that $a_{46}=1$ and set $x'_6=x_6+ \sum_i a_{4i} x_i$ so that the only monomial containing $x_4$ is $x_4x_6$. Moreover, we must have $a_{23}\neq 0$ for $\f_2$ to have rank 6. Rescaling adequately, we can assume $a_{23}=1$, so that
\begin{align*}
\f_2&=x_1x_5 + x_4 x_6 + x_2x_3 + x_2(a_{25}x_5+a_{26}x_6) + x_3(a_{35}x_5+a_{36}x_6) + a_{56}x_5x_6\\
&=(x_1+a_{25}x_2)x_5 + (x_4+a_{36}x_3) x_6 + x_2x_3 + a_{26}x_2x_6 + a_{35} x_3x_5 + a_{56}x_5x_6\,;
\end{align*}
the changes $x'_1=x_1+a_{25}x_2$, $x'_4=x_4+a_{36}x_3$ do not affect $\f_1$ and, resetting coordinates, we have $\f_2=x_1x_5 + x_2x_3 + x_4x_6 + a_{26}x_2x_6 + a_{35} x_3x_5 + a_{56}x_5x_6$. We impose the condition $\ell\cap\cC=\{\f_1\}$, which translates into $\f_2-\a\f_1$ having rank $6$ for every $\a \in \kk$. We compute
\begin{align*}
(\f_2-\a\f_1)^3 &= (x_1x_5 -\a x_1x_2 + x_2x_3 -\a x_3x_4 + x_4x_6 + a_{26}x_2x_6 + a_{35} x_3x_5 + a_{56}x_5x_6)^3\\
&=(a_{56}\a^2+(a_{26}+a_{35})\a-1)x_1x_2x_3x_4x_5x_6\,;
\end{align*}
since $\kk$ is algebraically closed, we must have $a_{26}=-a_{35}\eqqcolon a$ and $a_{56}=0$. This leaves us with and $\f_2=(x_1-ax_3)x_5+x_2x_3+(x_4+ax_2)x_6$. We change $x'_1=x_1-ax_3$, $x'_4=x_4+ax_2$, so that $\f_2=x'_1x_5+x_2x_3+x'_4x_6$, and $\f_1=(x'_1+ax_3)x_2+x_3(x'_4-ax_2)=x'_1x_2+x_3x'_4$. Resetting notation, we have arrived at the model
\[
\left\{\begin{array}{ccl}
 \f_1 & = & x_1x_2+x_3x_4\\
 \f_2 & = & x_1x_5+x_2x_3+x_4x_6
\end{array}\right.
\]

\medskip

\noindent \textbf{Case 4.} $\ell \cap \cC=\{p_1,p_2\}$ with multiplicities $2$ and $1$ respectively, and $\ell \cap \GG=\{p_1\}$. Since $p_1$ is a singular point for $\cC$, any line through $p_1$ has multiplicity of intersection $\ge 2$ with $\cC$, so $\ell$ is not contained in the tangent cone of $\cC$ at $p_1$. Consider the bivectors $\f_1=x_1x_2$, $\f_2=x_3x_4 + x_5x_6$ and the line $\ell=\la\a \f_1 + \b \f_2\ra$; points of $\ell$ with $\a, \b \ne 0$ have rank 6, hence $\ell\cap\cC=\{p_1,p_2\}$, with $p_i=[\f_i]$. As $\f_1$ is a singular point of $\cC$, $I_{p_1}(\ell,\cC)\geq 2$, so it must be $I_{p_1}(\ell,\cC)=2$ and $I_{p_2}(\ell,\cC)=1$. Of course this is also checked in coordinates: $\ell \cap \cC$ is given by solutions of $(\a \f_1+\b \f_2)^3=0$, i.e. $\a \b^2=0$.

\noindent To see that there is just one model with this relative position, take initial coordinates so that $\f_1=x_1x_2$, and note that $U_1 \cap U_2=\{0\}$, as their the sum must be the total space. Then we can write $U_2=\la x_3, x_4, x_5, x_6 \ra$, and the model is
\[
\left\{\begin{array}{ccl}
 \f_1 & = & x_1x_2\\
 \f_2 & = & x_3x_4+x_5x_6
\end{array}\right.
\]

\medskip

\noindent \textbf{Case 5.} $\ell \cap \cC=\{p\}$ with multiplicity $3$ and $\ell \cap \GG=\{p\}$. Recall that $p$ is a singular point of $\cC$ of multiplicity $2$ (as there are lines through $p$ intersecting $\cC$ with multiplicity $2$), so it follows that $\ell$ is contained in the tangent cone of $\cC$ at $p$.
Consider the bivectors $\f_1=x_1x_2$, $\f_2=x_1x_3+x_2x_4+x_5x_6$ and the line $\ell=\la\a \f_1 + \b \f_2\ra$; every point of $\ell$ with $\b\neq 0$ has rank $6$, so the point $p=[\f_1]$ is a triple point of intersection.

\noindent We aim at showing the uniqueness of this model. To see this, take initial coordinates with $\f_1=x_1x_2$. In the expression of $\f_2=\sum a_{ij} x_ix_j$ we can assume that $a_{12}=0$ by taking $\f'_2=\f_2-a_{12} \f_1$. Some coefficient $a_{1i}$, $i=3,4,5,6$ must be non-zero. We can make a permutation of $x_3,x_4,x_5, x_6$ so that $a_{13} \ne 0$, rescale so that $a_{13}=1$ and make a change $x'_3=x_3+ \sum a_{1i} x_i$ so that the only term containing $x_1$ is $x_1x_3$, and $\f_2=x_1x_3 + \xi_2$ with $\xi_2$ not containing $x_1$. The change $x'_1=x_1+a_{23}x_2$ eliminates the term $a_{23}x_2x_3$, hence we can suppose that at least one of $a_{24}$, $a_{25}$, $a_{26}$ is non-zero. After maybe permuting $x_4$ with either $x_5$ or $x_6$, and rescaling $x_4$, we can assume that $a_{24} =1$, make a change $x'_4=x_4+a_{25}x_5+a_{26}x_6$, and
\[
\f_2=x_1x_3+x_2x_4 + x_3(a_{34}x_4+a_{35}x_5+a_{36}x_6)+x_4(a_{45}x_5+a_{46}x_6)+a_{56}x_5x_6
\]
We claim that one of $a_{35}$, $a_{36}$, $a_{45}$, $a_{46}$ must be non-zero. Indeed, if all of them were zero then we could rescale to achieve $a_{34}=1=a_{56}$, and then we would have $\f_1+\f_2=(x_1-x_4)(x_2+x_3)+x_5x_6$ of rank $4$, which is a contradiction. Moreover we can assume that $a_{35} \ne 0$, maybe after permuting the pairs $(x_1,x_3)$ and $(x_2,x_4)$, and also permuting $x_5, x_6$ if necessary. Rescale so that $a_{35}=1$, do the change $x'_5=x_5+a_{34}x_4+a_{36}x_6$ and reset notation to get
\[
\f_2=x_1x_3+x_2x_4 + x_3x_5+x_4(a_{45}x_5+a_{46}x_6)+a_{56}x_5x_6\,.
\]
We compute $(\f_2+\a \f_1)^3=-6(a_{46} \a +a_{56})x_1x_2x_3x_4x_5x_6$; since this must be non-zero for every $\alpha\in\kk$, it must be $a_{46}=0$ and $a_{56} \ne 0$; we can rescale to have $a_{56}=1$ and get
\[
\f_2=x_1x_3+x_2x_4+x_5(x_6-x_3-a_{45}x_4)=x_1x_3+x_2x_4+x_5x'_6\,.
\]
With a last change $x'_6=x_6-x_3-a_{45}x_4$ this gives the model
\[
\left\{\begin{array}{ccl}
 \f_1 & = & x_1x_2\\
 \f_2 & = & x_1x_3+x_2x_4+x_5x_6
\end{array}\right.
\]
\end{proof}

\begin{remark}
An alternative method to prove the uniqueness of the model for each relative position is based on the classification of pencils of skew-symmetric matrices, which can be found in \cite{Gant}. A pencil of skew-symmetric matrices can be thought of as a line $\ell$ of bivectors \emph{with two marked points}, the generators of the pencil, so the classification of pencils consists of finding standard models for pairs of bivectors (or skew-symmetric matrices). This is more rigid that the classification of lines $\ell$ we do here, as in the latter case we are allowed to vary the generators of the line.
\end{remark}

\noindent In Table \ref{table:6-2C} we collect the results of Proposition \ref{prop:case62-rel-positions-l-not-contained-C}, which tackled the case $\ell\not\subset\cC$.
\begin{itemize}
    \item The second column contains the relative position of $\ell$ with respect to $\GG\subset\cC$;
    \item the third column contains the relative position of $\ell$ with respect to $\cC$;
    \item the fourth and fifth columns contain the differentials of the non-closed elements;
    \item the sixth column says whether the minimal algebra is {\em irreducible}, i.e.~it is not the sum of lower-dimensional minimal algebras;
    \item in case it is irreducible, the seventh column identifies our algebra with the Lie algebra in the list obtained in \cite{Ren2011}.
\end{itemize}

\begin{table}[h]
\caption{Minimal algebras of type $(6,2)$ with $\ell\not\subset\cC$}\label{table:6-2C}
{\tabulinesep=1mm
\begin{tabu}{|c|c|c|c|c|c|c|c|c|c|c|c|c|c|}
\hline
Label & $\ell\cap\GG$ & $\ell\cap\cC$ & $dx_7$ & $dx_8$ & Irreducible & \cite{Ren2011}\\
\hline
(6.2.7) & $\emptyset$ & $\{p_1,p_2,p_3\}$ & $x_1x_2+x_3x_4$ & $x_3x_4+x_5x_6$ & $\checkmark$ & $N_1^{8,2}$\\
\hline
(6.2.8)& $\emptyset$ & $\{2p_1,p_2\}$ & $x_1x_2+x_3x_4$ & $x_1x_5+x_3x_6$ & $\checkmark$ & $N_3^{8,2}$\\
\hline
(6.2.9) & $\emptyset$ & $\{3p\}$ & $x_1x_2+x_3x_4$ & $x_1x_5+x_2x_3+x_4x_6$ & $\checkmark$ & $N_5^{8,2}$\\
\hline
(6.2.10) & $1$ & $\{2p_1,p_2\}$ & $x_1x_2$ & $x_3x_4+x_5x_6$ & $\times$ & \\
\hline
(6.2.11) & $1$ & $\{3p\}$ & $x_1x_2$ & $x_1x_3+x_2x_4+x_5x_6$ & $\checkmark$ & $N_4^{8,2}$\\
\hline
\end{tabu}}
\end{table}

\section{Case (5,3)}\label{sec:(5,3)}

We have $d\colon F_1 \to \L^2 W_0$ injective, with $\dim F_1=3$, $\dim W_0=5$, and $\pi=\PP(\tr{Im}(d))$ is a projective $2$-plane in $\PP^9=\PP(\L^2 W_0)$. Denote again $W_0=W$. As every bivector in $\L^2 W$ has rank at most $4$, the stratification by rank has only one stratum, namely the rank-$2$ bivectors given by the Plücker embedding of the Grassmannian $\tr{Gr}(2,5)$ of planes in $W \cong \kk^5$, or, equivalently, of the Grassmannian $\GG(1,4)$ of projective lines in $\PP(W) \cong \PP^4$. We set $\GG \coloneqq \GG(1,4)$ in this section. We need to study the relative position of $\pi$ with respect to $\GG$ in $\PP^9$. In order to do so, recall that the image of the Plücker embedding of $\GG$ in $\PP^9$ is a variety of dimension $6$ and degree $5$ (see \cite{Harris}). A bivector $\f$ is in $\GG$ if and only if $\f^2=0$; in coordinates, 
$\f=\sum_{i<j} a_{ij} x_i x_j$ and
\begin{align*} 
    \f^2= & \,  (a_{12}a_{34}-a_{13}a_{24}+a_{14}a_{23})x_1x_2x_3x_4+(a_{12}a_{35}-a_{13}a_{25}+a_{15}a_{23})x_1x_2x_3x_5  \\  
    + & (a_{12}a_{45}-a_{14}a_{25}+a_{15}a_{24})x_1x_2x_4x_5 +(a_{13}a_{45}-a_{14}a_{35}+a_{15}a_{34})x_1x_3x_4x_5 \\ 
    + & (a_{23}a_{45}-a_{24}a_{35}+a_{25}a_{34})x_2x_3x_4x_5\,,
\end{align*}
hence $\GG$ is given by the $5$ equations obtained by equating the above coefficients to zero:
\begin{equation}\label{eq:G(2,5)}
\GG=\begin{cases}
a_{12}a_{34}-a_{13}a_{24}+a_{14}a_{23}=0 \\
a_{12}a_{35}-a_{13}a_{25}+a_{15}a_{23}=0 \\
a_{12}a_{45}-a_{14}a_{25}+a_{15}a_{24}=0 \\
a_{13}a_{45}-a_{14}a_{35}+a_{15}a_{34}=0 \\
a_{23}a_{45}-a_{24}a_{35}+a_{25}a_{34}=0
\end{cases}    
\end{equation}

\noindent We want to obtain a representative of the orbit of the plane $\pi \subset \L^2 W$ in terms of generators.
We call $\f_1, \f_2, \f_3$, some generators of $\pi$, and $x_1, x_2,x_3,x_4,x_5$ a basis for $W$. As before, the idea is to choose $\f_i$ and $x_i$ so that the expression is as simple as possible.

\begin{proposition} \label{prop:case53-rel-positions-pi-contained-G}
Notations as above. Suppose that $\pi\subset\GG$. Then one and only one of the following occurs:
\begin{enumerate}
\item $\pi=\PP(\Lambda^2 Z)$ for some $Z\subset W$ with $\dim Z=3$. We can take generators for $\pi$ of the form $\f_1=x_1x_2$, $\f_2=x_1x_3$, and $\f_3=x_2x_3$.

\medskip

\item $\pi\subset\PP(\Lambda^2 Y)\cong\PP^5$ for some $Y\subset W$ with $\dim Y=4$. We can take generators for $\pi$ of the form $\f_1=x_1x_2$, $\f_2=x_1x_3$, and $\f_3=x_1x_4$.
\end{enumerate}
Each of the above relative positions determines a standard model for $\pi$ in suitable coordinates, hence the orbit of $\pi$ under the action of $\GL(W)$. 
\end{proposition}

\begin{proof}
Choose rank-2 generators $\f_i$, $i=1,2,3$, and note that the planes $U_i\subset\kk^5$ have to satisfy $\dim(U_i \cap U_j)=1$ for $i \ne j$. Indeed, if this were not the case, linear combinations of $\f_i$ and $\f_j$ would have rank $4$. Then
\[
\dim(U_1+U_2+U_3)=3+\dim(U_1\cap U_2\cap U_3)\,.
\]

\item If $\dim(U_1+U_2+U_3)=3$, we can choose coordinates in $Z\coloneqq(U_1+U_2+U_3)$ which give the model
\[
\left\{\begin{array}{ccl}
 \f_1 & = & x_1x_2\\
 \f_2 & = & x_1x_3\\
 \f_3 & = & x_2x_3
\end{array}\right.
\]

\item If $\dim(U_1+U_2+U_3)=4$, we choose $x_1$ spanning $U_1\cap U_2\cap U_3$ and complete it to a basis $\{x_1,x_2,x_3,x_4\}$ of $Y\coloneqq(U_1+U_2+U_3)$ which gives the model
\[
\left\{\begin{array}{ccl}
 \f_1 & = & x_1x_2\\
 \f_2 & = & x_1x_3\\
 \f_3 & = & x_1x_4
\end{array}\right.
\]
\end{proof}

\noindent We assume now that $\pi$ is not contained in $\GG$. We start with the case in which $\pi \subset \PP(\L^2Y)$, for some $Y \subset W$ with $\dim Y=4$. These $\PP(\L^2 Y) \cong \PP^5$ are special 5-dimensional subspaces of $\PP(\L^2 W) \cong \PP^9$ for the action of $\GL (W)$ in $\PP^9$. In particular, the orbit $\GL (W) \cdot \PP(\L^2 Y)$ is a subvariety of the Grassmannian of $5$-dimensional subspaces of $\PP^9$. 

\noindent In practical terms, the condition $\pi \subset \PP(\L^2 Y)\cong\PP^5$ says that we only need four vectors to describe the bivectors $\f_1, \f_2, \f_3$. Note that if $\pi \not \subset \GG$ then it necessarily intersects $\GG$ in a conic $\mathscr{C}$, because $\GG\cap\PP^5$ is the Klein quadric in $\PP^5$. The equation for the Klein quadric is obtained by putting $\f^2=0$ for $\f$ a bivector in $\L^2 Y$, with $Y=\la x_1,x_2,x_3,x_4 \ra$. The coordinate $x_5$ does not appear in any differential.
Hence, this case can be reduced to dimension seven, and is handled in \cite{B}. In other words, the minimal algebras arising in this way are a direct sum $\L  V_1  \oplus \L V_2$, where $V_1=\la x_5 \ra$ is a one-dimensional subspace, $V_2$ has dimension $7$ and the minimal algebra $\L V_2$ is one of the models from \cite{B} with $dx_6=\f_1$, $dx_7=\f_2$, $dx_8=\f_3$.
For completeness, let us briefly review those models.
\begin{itemize}
\item Assume that $\mathscr{C}$ is a smooth conic and choose $\f_1, \f_2 \in \mathscr{C}$. Since $\mathscr{C}$ is smooth, the line $\la \f_1, \f_2 \ra$ does not have further intersection points with $\GG$; it follows that $U_1\cap U_2=0$, and we can find coordinates so that $\f_1=x_1x_2$, $\f_2=x_3x_4$. We choose $\f_3$ as the intersection of the tangent lines $T_{\f_1}\mathscr{C} \cap T_{\f_2}\mathscr{C}$. Then $\f_3\notin\mathscr{C}$, so it has rank $4$, and has the form 
\[
\f_3=x_1(ax_2+bx_3+cx_4)+x_2(ex_3+fx_4)+gx_3x_4
\]
Now, $\f_3-a\f_1$ also has rank $4$ since, by our choice of $\f_3$, the only intersection point of the line $\la \f_1, \f_3 \ra$ with $\mathscr{C}$ is $\f_1$. Hence $(\f_3-a\f_1)^2\neq 0$, and it follows from this that $bf-ce \ne 0$, so we can make the change $x'_3=bx_3+cx_4$, $x'_4=ex_3+fx_4$. This change preserves $\f_1$ and $\f_2$ (up to scalars) and gives
\[
\f_3=ax_1x_2 + x_1x'_3+x_2x'_4+ gx'_3x'_4
\]
so we can take the new generator $\f'_3=\f_3-a\f_1-g\f_2$ to obtain the model
\[ 
\left\{\begin{array}{ccl}
 \f_1 & = & x_1x_2\\
 \f_2 & = & x_3x_4\\
 \f_3 & = & x_1x_3+x_2x_4
\end{array}\right.
\]
In this model, in coordinates $[\a: \b: \g]$ with respect to $\f_1$, $\f_2$ and $\f_3$, we have $\pi \cap \GG=\{\a \b - \g^2=0\}$. An equivalent model for this case is
\[ 
\left\{\begin{array}{ccl}
 \tilde{\f}_1 & = & x_1x_2\\
 \tilde{\f}_2 & = & x_3x_4\\
 \tilde{\f}_3 & = & (x_1+x_3)(x_2+x_4)
\end{array}\right.
\]
which stems from another choice of third generator. It can be obtained by the one above by considering $\f'_3=\f_1+\f_2+\f_3=(x_1-x_4)(x_2+x_3)$ and then permuting $x_3$ and $x_4$. 

\medskip

\item Assume $\mathscr{C}=\ell_1 \cup \ell_2$ is a pair of distinct lines. In this case we take $\f_2$ as the point of intersection of the lines, $\f_1 \in \ell_1$, $\f_3 \in \ell_2$, and coordinates so that $U_1=\la x_1,x_2 \ra$, $U_2=\la x_1, x_3 \ra$, $U_3=\la x_3,x_4 \ra$, so
\[ 
\left\{\begin{array}{ccl}
 \f_1 & = & x_1x_2\\
 \f_2 & = & x_1x_3\\
 \f_3 & = & x_3x_4
\end{array}\right.
\]
In this model, $\pi \cap \GG=\{\a \g=0\}$.

\medskip

\item Assume $\mathscr{C}$ is a double line; take $\f_1, \f_2$ on the line, and $\f_3$ outside, so it has rank 4; we take coordinates so that $\f_1=x_1x_2$, $\f_2=x_1x_3$, and we can write
\[
\f_3=x_1(ax_2+bx_3+cx_4)+x_2(ex_3+fx_4)+gx_3x_4\,.
\]
Now, if $f \ne 0$ then we can assume $f=1$ rescaling $x_4$ and then make the change $x'_4=ex_3+x_4$. Clearly, $\f_3=x_1(ax_2+bx_3+cx'_4)+x_2x'_4+gx_3x'_4$, so $\f'_3=\f_3-a\f_1-b\f_2$ has rank $2$, a contradiction. We deduce $f=0$. 
Now, if $g \ne 0$ we can assume $g=1$ and consider the change $x'_4=x_4-ex_2$, so $\f_3=x_1(ax_2+bx_3+cx'_4)+x_3 x'_4$, and we again obtain a contradiction considering $\f'_3=\f_3-a\f_1-b\f_2$.
Hence we must have $f=g=0$, so $e \ne 0$ and we can rescale $x_2$ so $e=1$, so $\f_3=x_1x'_4+x_2x_3$ with the change $x'_4=ax_2+bx_3+cx_4$, and $c \ne 0$ since $U_3$ has dimension $4$. We have the model 
\[ 
\left\{\begin{array}{ccl}
 \f_1 & = & x_1x_2\\
 \f_2 & = & x_1x_3\\
 \f_3 & = & x_1x_4+x_2x_3
\end{array}\right.
\]
\end{itemize}

\noindent The models obtained so far are characterized by the property that $\pi\subset\PP(\Lambda^2 U)$, for $U\subset W$ a subspace of dimension $\leq 4$. We collect them in Table \ref{table:5-3_1}. Notice that these minimal algebras are reducible.
\begin{itemize}
    \item The second column contains the intersection of $\pi$ and the Grassmannian $\GG=\GG(1,4)$;
    \item The thirds column contains the dimension $\delta$ of a subspace $\PP(\L^2 U) \subset \PP^9=\PP(\L^2 W)$, for $U\subset W$ a subspace of dimension $\leq 4$, in which $\pi$ is contained. Clearly $\delta\in\{2,5\}$;
    \item the fourth, fifth, and sixth columns contain the differentials of the non-closed elements.
\end{itemize}

\begin{table}[h]
\caption{Minimal algebras of type $(5,3)$ with $\pi\subset\PP^5$}\label{table:5-3_1}
{\tabulinesep=1mm
\begin{tabu}{|c|c|c|c|c|c|c|c|c|c|c|c|c|c|}
\hline
Label & $\pi\cap\GG$ & $\delta$ & $dx_6$ & $dx_7$ & $dx_8$\\
\hline
(5.3.1) & $\pi$ & 2 & $x_1x_2$ & $x_1x_3$ & $x_2x_3$\\
\hline
(5.3.2) & $\pi$ & 5 & $x_1x_2$ & $x_1x_3$ & $x_1x_4$\\
\hline
(5.3.3) & smooth conic & 5 & $x_1x_2$ & $x_3x_4$ & $x_1x_3+x_2x_4$\\
\hline
(5.3.4) & pair of lines & 5 & $x_1x_2$ & $x_1x_3$ & $x_3x_4$\\
\hline
(5.3.5) & double line & 5 & $x_1x_2$ & $x_1x_3$ & $x_1x_4+x_2x_3$\\
\hline
\end{tabu}}
\end{table}


\noindent We move now to the case in which $\pi \not \subset \GG$ and $\pi \not \subset \PP(\L^2 U)$, for $U\subset W$ a 4-dimensional subspace. We start with an auxiliary result. 

\begin{proposition} \label{prop:pi-must-contain-line}
Suppose $\pi \subset \PP^9$ is a plane not contained in any $\PP^5=\PP(\L^2 U)$, where $U\subset W$ is a 4-dimensional subspace. If $\pi\cap\GG$ contains at least 4 points, then it contains a line.
\end{proposition}
\begin{proof}
Let $\f_1$, $\f_2$ and $\f_3$ be generators of $\pi$, which we can assume to be in $\pi\cap\GG$. Assume by contradiction that $\pi \cap \GG$ contains no lines. This means that the lines $\la \f_i, \f_j \ra$ are not contained in $\GG$, so that $U_i \cap U_j=\{0\}$ for $i \ne j$. Therefore, we can take coordinates so that $U_1=\la x_1, x_2 \ra$, $U_2=\la x_3, x_4 \ra$, and $U_3=\la x_5, ax_1+bx_2+cx_3+dx_4 \ra$, with $(a,b) \ne (0,0) \ne (c,d)$.


\item Assuming $a \ne 0 \ne c$, we put $x'_1=ax_1+bx_2$, $x'_3=cx_3+dx_4$ and we get
\[
\f_1=x_1x_2 \,, \qquad \f_2=x_3x_4 \qquad \textrm{and}
\qquad \f_3=x_5(x_1+x_3)\,.
\]
It is a straightforward computation to see that all bivectors in $\pi$ have rank $4$ except for the three generators $\f_1$, $\f_2$ and $\f_3$, so $\pi \cap \GG$ consists of three points. This gives a contradiction, hence $\pi$ contains a line.
\end{proof}

\noindent Suppose that $\pi$ is a plane not contained in any $\PP(\L^2 U)$ with $U\subset W$, $\dim U=4$, and that $|\pi\cap\GG|\geq 4$. Then $\pi \cap \GG$ contains a line by Proposition \ref{prop:pi-must-contain-line}. If we take two generators $\f_1$, $\f_2$ on this line, we can choose coordinates $x_1,x_2,x_3$ so that $\f_1=x_1x_2$, $\f_2=x_1x_3$. We handle the third generator according to different possibilities.

\noindent \textbf{Case 1.} There exists a third point $\f_3\in\pi\cap\GG$ with rank $2$. In this case, we must have that $U_3 \oplus \la x_1, x_2, x_3 \ra=W$, as $\pi$ is not contained in any special $\PP^5$. We can arrange so that $\f_3= x_4x_5$, and we have the model
\[ 
\left\{\begin{array}{ccl}
 \f_1 & = & x_1x_2\\
 \f_2 & = & x_1x_3\\
 \f_3 & = & x_4x_5
\end{array}\right.
\]
We see that $\pi \cap \GG=\ell \cup \{p\}$ is a line plus a point. In coordinates $[\a: \b: \g]$ with respect to $\f_1$, $\f_2$ and $\f_3$, we see that $\pi \cap \GG=\{\a \g=0 \, , \b \g =0 \}$,
so $\pi \cap \GG$ is indeed a simple line with an extra point, a scheme of dimension $1$ and degree $1$.

\begin{corollary} 
In the same hypotheses as in Proposition \ref{prop:pi-must-contain-line}, either $\pi\cap\GG=\ell$, or $\pi\cap\GG=\ell\cup\{p\}$, with $p\notin\ell$.
\end{corollary}
\begin{proof}
By Proposition \ref{prop:pi-must-contain-line} we know that $\pi\cap\GG$ contains a line. Case 1 above and Case 2 below show that both $\pi\cap\GG=\ell$ and $\pi\cap\GG=\ell\cup\{p\}$ for $p\notin\ell$ can happen. Case 1 shows that, in presence of a point $p$ in $\pi\cap\GG$ but not on $\ell$, the (scheme-theoretic) intersection $\pi\cap\GG$ is $\ell\cup\{p\}$.
\end{proof}

\noindent In the remaining cases, any third generator $\f_3$ has rank $4$, so $\pi \cap \GG$ is the line $\la\f_1,\f_2\ra$.

\noindent \textbf{Case 2.} There exists $\f_3$ such that $U_1 \cap U_2 \subset U_3$. In this case we claim that we can choose generators of $\pi$, and a basis for $W$, so that
\[ 
\left\{\begin{array}{ccl}
 \f_1 & = & x_1x_2\\
 \f_2 & = & x_1x_3\\
 \f_3 & = & x_1x_4+x_2x_5
\end{array}\right.
\]
Suppose $U_1\cap U_2=\la x_1\ra$, so that $x_1 \in U_3$, and also $x_4, x_5 \in U_3$ since $\pi$ is not contained in any special $\PP^5$. Hence $U_3$ must be generated by $x_1, x_4, x_5$ and $ax_2+bx_3$ with $a$ or $b$ non zero. We can assume $a \ne 0$ so rescaling we get $a=1$. Making the change $x'_2=x_2+bx_3$, considering the new generator $\f'_1=\f_1+b\f_2=x_1x'_2$ and resetting the notation we have $U_3=\la x_1, x_2,x_4,x_5\ra$. 

\noindent In view of this $\f_3$ must have the form $\f_3=x_1(cx_4+ex_5) + x_2(fx_4+gx_5) + hx_4x_5$, with $cg-ef \ne 0$ so we can do the change $x'_4=cx_4+ex_5$, $x'_5=fx_4+gx_5$ and reset notation to obtain $\f_3=x_1x_4+x_2x_5+gx_4x_5$. If $h$ is non-zero we can assume $h=1$ by rescaling $x_5$, $x_2$ and $\f_1$. But then $\f_3+\f_1=x_1(x_2+x_4)+(x_2+x_4)x_5$ has rank $2$, a contradiction with the assumption that $\pi \cap \GG$ is a line. We deduce $g=0$, and we are done.

\noindent If we study the scheme-theoretic intersection $X\coloneqq \pi \cap \GG$, we get
$X=\mathrm{proj}(\kk[x,y,z]/(yz,z^2))$, whose Hilbert function is $h(n)=n+2$, hence $X$ has dimension $1$ and degree $1$, so it is an ordinary line.

\noindent\textbf{Case 3.} For any third generator $\f_3$ we have that $U_1 \cap U_2 \cap U_3=\{0\}$ and the annihilator $(U_1+U_2)^0$ is an isotropic plane for $\f_3$, i.e.~$\f_3$ vanishes there. We claim that we can choose generators of $\pi$ so that
\[ 
\left\{\begin{array}{ccl}
 \f_1 & = & x_1x_2\\
 \f_2 & = & x_1x_3\\
 \f_3 & = & x_2x_4+x_3x_5
\end{array}\right.
\]
Let us check first that this model satisfies the requirement: any other third generator has the form
\[
\f'_3=\a \f_1+\b \f_2+\f_3=\a x_1x_2+\b x_1x_3+x_2x_4+x_3x_5\,,
\]
hence $U'_3$, the 4-plane associated to $\f'_3$, is $\la \a x_2+\b x_3,-\a x_1+x_4,-\b x_1+x_5,-x_2,-x_3 \ra$ and does not contain $U_1 \cap U_2=\la x_1 \ra$ (if it did, it would be $5$-dimensional). It is also clear that $\f'_3$ vanishes in the plane
$\langle v_4,v_5\rangle$, where $\{v_i\}$ is the basis dual to $\{x_i\}$.

\noindent Let us see why we can always choose coordinates as above. Take any third generator $\f_3$, and note that $U_3 \cap U_i$ must be a line for $i=1,2$. Take a basis for $W$ so that $U_3\cap U_1=\la x_2 \ra$, $U_3 \cap U_2=\la x_3 \ra$, and $U_3=\la x_2,x_3,x_4,x_5 \ra$. Then
\begin{equation} \label{eq:fi3}
\f_3=x_2(ax_3+bx_4+cx_5)+x_3(ex_4+fx_5)+gx_4x_5\,;
\end{equation}
notice that $g=0$ since, by assumption, $\f_3$ vanishes in $(U_1+U_2)^0=\la v_4,v_5 \ra$. It follows that one of $b$, $c$ must be non-zero (otherwise $\f_3$ has rank $2$). We can assume that $b \ne 0$, so rescale it to get $b=1$ and do the change $x'_4=ax_3+x_4+cx_5$, so $\f_3=x_2x'_4+x_3(ex'_4+fx_5)$, and now we see that it must be $f \ne 0$, so we can assume $f=1$ and change $x'_5=ex'_4+x_5$ and we are done.

\noindent If $X=\pi \cap \GG$, an easy calculation gives $X=\mathrm{proj}(\kk[x,y,z]/(xz, yz, z^2))$, whose Hilbert function is $h(1)=3$ and $h(n)=n+1$ for $n \ge 2$, hence $X$ has dimension $1$ and degree $1$, so it is an ordinary line again.

\noindent\textbf{Case 4.} For any third generator we have $U_1 \cap U_2 \cap U_3=\{0\}$ and $\f_3$ is non-degenerate in the annihilator $(U_1+U_2)^0$. We claim that we can choose generators for $\pi$ of the form:
\[ 
\left\{\begin{array}{ccl}
 \f_1 & = & x_1x_2\\
 \f_2 & = & x_1x_3\\
 \f_3 & = & x_2x_3+x_4x_5
\end{array}\right.
\]
We check first that this model satisfies the condition: any third generator has the form 
\[
\f'_3=a \f_1+b\f_2+\f_3=ax_1x_2+bx_1x_3+x_2x_3+x_4x_5
\]
so $U'_3=\langle ax_2+bx_3,-ax_1+x_3,-bx_1-x_2,x_5,-x_4 \rangle$ does not contain $U_1 \cap U_2=\langle x_1 \rangle$ (as above, it would be 5-dimensional otherwise). Also, $\f'_3$ is non-degenerate in $(U_{\f_1}+U_{\f_2})^0=\langle v_4,v_5\rangle$.

\noindent Now we show how to get the above model. Take a third generator $\f_3$ and choose coordinates so that $U_3=\la x_2,x_3, x_4, x_5\ra$, as was done in the previous paragraph. Now the form of $\f_3$ is as in \eqref{eq:fi3} but with $g \ne 0$, so we can assume $g=1$ and write
\[
\f_3=ax_2x_3+cx_2x_5+fx_3x_5+x_4(x_5-bx_2-ex_3)
\]
and put $x'_5=x_5-bx_2-ex_3$ so that (after resetting the notation) $\f_3=ax_2x_3+cx_2x_5+fx_3x_5+x_4x_5$, with $a\ne 0$ since $\f_3$ has rank $4$. We rescale $x_2$ and $\f_1$ so that get $a=1$. Reset notation and write $\f_3=x_2x_3+(cx_2+fx_3+x_4)x_5$, and the change $x'_4=x_4+cx_2+fx_3$ yields the desired model. Finally, the intersection $X=\pi \cap \GG$ is $X=\mathrm{proj}(\kk[x,y,z]/(xz, yz, z^2))$, which is isomorphic, as a scheme, to the one obtained in Case 3.

\noindent In order to give a more intrinsic characterization of the above cases, denote by $\tau$ the restriction to $\la\f_1,\f_2,\f_3\ra\subset\Lambda^2 W$ of the linear map $\Lambda^2 W\to\Lambda^4 W$, $\f\mapsto\f\wedge\f_3$, followed by the isomorphism $\Lambda^4 W\to W^*\otimes\Lambda^5 W$.
\begin{itemize}
    \item $\tr{Im}(\tau)$ has dimension 2 in Case 1, 1 in Case 2 and 2 in Cases 3 and 4;
    \item $\tr{Im}(\tau)$ is an isotropic subspace for $\f_3$ in Case 3 , while the restriction of $\f_3$ to $\tr{Im}(\tau)$ is non-degenerate in Case 4. 
\end{itemize}

\noindent The last four models are characterized by the property that $\pi\not\subset\PP(\Lambda^2 U)$, for any proper subspace $U\subset W$, and $\pi\cap\GG$ contains a line. We collect these results in Table \ref{table:5-3_2}.
\begin{itemize}
    \item The second column contains the intersection of $\pi$ and the Grassmannian $\GG=\GG(1,4)$;
    \item The thirds column contains the description of $\tr{Im}(\tau)$;
    \item the fourth, fifth, and sixth columns contain the differentials of the non-closed elements;
    \item the seventh column says whether the minimal algebra is {\em irreducible}, i.e.~it is not the sum of lower-dimensional minimal algebras;
    \item in case it is irreducible, the eighth column identifies our algebra with the Lie algebra in the list obtained in \cite{Ren2011}.
\end{itemize}

\begin{table}[h]
\caption{Minimal algebras of type $(5,3)$ with $\pi\not\subset\PP^5$, $\pi\cap\GG$ contains a line}\label{table:5-3_2}
{\tabulinesep=1mm
\begin{tabu}{|c|c|c|c|c|c|c|c|c|c|c|c|c|c|}
\hline
Label & $\pi\cap\GG$ & $\tr{Im}(\tau)$ & $dx_6$ & $dx_7$ & $dx_8$ & Irreducible & \cite{Ren2011}\\
\hline
(5.3.6) & $\ell\cup\{p\}$ & 2-dimensional & $x_1x_2$ & $x_1x_3$ & $x_4x_5$ & $\times$ & \\
\hline
(5.3.7) & $\ell$ & 1-dimensional & $x_1x_2$ & $x_1x_3$ & $x_1x_4+x_2x_5$ & \checkmark & $N_5^{8,3}$\\
\hline
(5.3.8) & $\ell$ & 2d, isotropic & $x_1x_2$ & $x_1x_3$ & $x_2x_4+x_3x_5$ & \checkmark & $N_2^{8,3}$\\
\hline
(5.3.9) & $\ell$ & 2d, non-degenerate & $x_1x_2$ & $x_1x_3$ & $x_2x_3+x_4x_5$ & \checkmark & $N_1^{8,3}$\\
\hline
\end{tabu}}
\end{table}

\noindent Next, we tackle the case in which $X\coloneqq\pi\cap \GG$ does not contain a line. In view of Proposition \ref{prop:pi-must-contain-line}, this amounts to $|X|\leq 3$.



\noindent Let us deal first with the case $|X|=3$. Consider three points $\f_1$, $\f_2$ and $\f_3$ in $X$, and note that $U_i \cap U_j=\{0\}$ for $i \ne j$, since otherwise the line generated by $\f_i, \f_j$ would be in $X$. Hence we can choose a basis for $W$ such that $\f_1=x_1x_2$, $\f_2=x_3x_4$, and $\f_3=x_5 (ax_1 + bx_2 + cx_3 + ex_4)$, with $(a,b) \ne (0,0) \ne (c,e)$. With a change of coordinates as in Proposition \ref{prop:pi-must-contain-line} we get $\f_3=x_5(x_1+x_3)$, and now it is easy to see that every linear combination $\f=\a \f_1 + \b \f_2+ \g \f_3$ has rank $4$ except for $\f_1, \f_2, \f_3$, so $X$ does not contain any fourth point. We get the model
\[ 
\left\{\begin{array}{ccl}
 \f_1 & = & x_1x_2\\
 \f_2 & = & x_3x_4\\
 \f_3 & = & x_1x_5+x_3x_5
\end{array}\right.
\]
Let us study $X \subset \pi$ as a subvariety. Points of $\pi$ are parameterized as $\{a_{12}=\a,a_{34}=b,a_{15}=a_{35}=\gamma\}$; plugging this into \eqref{eq:G(2,5)} we get 
$X=\{
\a \b=\a \g=\b \g=0\}$, so $X$ is a three-points set. As a scheme, $X=\mathrm{Proj}(\kk[x,y,z]/(xy,xz,yz))$ has Hilbert function $h(n)=3$, so it has dimension $0$ and degree 3. This confirms that $X$ is a three-points scheme. 

\noindent We consider now the case $|X|=2$. We call $\f_1, \f_2$ the points in $X$. Note that $U_{12}\coloneqq U_1 + U_2$ must have dimension $4$. If $\f\in\pi$ is not collinear with $\f_1, \f_2$, $U_\f$ has dimension $4$ and cannot contain both $U_1$ and $U_2$, for otherwise $\pi$ would be contained in a special $\PP^5\cong\PP(\L^2 U_{\f})$. Hence $\dim(U_\varphi+U_{12})=5$, $\dim(U_{\f} \cap U_{12})=3$ and $1\leq \dim(U_{\f} \cap U_i)\leq 2$, for $i=1,2$. We have two cases:
\[
U_1\subset U_{\f}\,, \ \dim(U_{\f} \cap U_2)=1 \qquad \tr{and} \qquad \dim(U_{\f} \cap U_i)=1\,, i=1,2\,.
\]
\noindent As we shall see a posteriori, those properties can also be distinguished by studying $X$ scheme-theoretically, according to the existence of points with multiplicity.

\noindent \textbf{Case 1.} $\pi$ has a third generator $\f_3$ such that $U_1\subset U_3$ and $\dim(U_3 \cap U_2)=1$. We shall see that one can choose coordinates so that $\pi=\la \f_1,\f_2,\f_3\ra$ with
\[ 
\left\{\begin{array}{ccl}
 \f_1 & = & x_1x_2\\
 \f_2 & = & x_3x_4\\
 \f_3 & = & x_1x_3+x_2x_5
\end{array}\right.
\]
Indeed, we arrange first that $U_1=\la x_1, x_2 \ra$, $U_2=\la x_3, x_4 \ra$ and $U_3=\la x_1, x_2, x_3, x_5 \ra$, so $\f_3$ has the form
\[
x_1(ax_3+bx_5)+x_2(cx_3+ex_5)+fx_3x_5
\]
with one of $b, e \ne 0$ since $\f_3$ has rank $4$. By swapping $x_1$ and $x_2$ if necessary, we assume $e\ne 0$, so $e=1$ rescaling $x_5$; we make the change $x'_5=x_5+cx_3$. Upon resetting notation, 
\[
\f_3=x_1(ax_3+bx_5)+x_2x_5+fx_3x_5=ax_1x_3+(bx_1+x_2)x_5+fx_3x_5 \,.
\]
It must be $a \ne 0$, so we assume $a=1$ by rescaling $x_1$ and $\f_1$. Make the change $x'_2=x_2+bx_1$, so $\f_3=x_1x_3+(x_2+fx_3)x_5$. If $f \ne 0$ then we could assume $f=1$ by rescaling $x_3$ and $x_1$ (and $\f_2$, $\f_1$ accordingly), but then $\f_3+\f_1$ would have rank $2$, a contradiction. We deduce $f=0$ and we are done.

\noindent As a scheme, $X=\mathrm{Proj}(\kk[x,y,z]/(xy,z^2,yz))$. Hence $X$ is a two-points set with Hilbert function $h(n)=3$, hence $X$ has dimension $0$ and degree $3$; we deduce that one of the two points is double. Indeed, in our coordinates $p=[1:0:0]$ is a double point: take the affine chart $A=\{x \ne 0\}$, in which $p=(0,0)$ and $X|_A \cong \mathrm{Spec}(\kk[z]/(z^2))$.

\noindent \textbf{Case 2.} Any third generator $\f_3$ of $\pi$ satisfies that $U_3$ intersects both $U_1$ and $U_2$ in a line. In this case we can choose coordinates so that $\pi=\la \f_1,\f_2,\f_3\ra$ with
\[ 
\left\{\begin{array}{ccl}
 \f_1 & = & x_1x_2\\
 \f_2 & = & x_3x_4\\
 \f_3 & = & x_1x_3+(x_2+x_4)x_5
\end{array}\right.
\]
First note that the plane given above satisfies the requirement: another third generator $\f'_3=\a \f_1+\b \f_2+\f_3$ has associated subspace $U'_3=\la\a x_2+x_3,\a x_1-x_5,x_1-\b x_4,x_5-\b x_3,x_2+x_4\ra$, which is easily seen to be 4-dimensional. Moreover, if $U_1$ or $U_2$ were contained in $U'_3$, then $U'_3$ would be $5$-dimensional, a contradiction.

\noindent Let us show how to choose coordinates to obtain the claimed model. We first arrange that $U_1=\la x_1, x_2 \ra$, $U_2=\la x_3, x_4 \ra$, $U_3=\la x_1, x_3, x_5, x_2+x_4 \ra$ with the usual argument. Then
\begin{align}\label{eq:5:3:1}
\f_3&=ax_1x_3+bx_1x_5+cx_1(x_2+x_4)+ex_3x_5+fx_3(x_2+x_4)+gx_5(x_2+x_4)\nonumber\\
&=ax_1x_3+(cx_1+fx_3+gx_5)(x_2+x_4)+(bx_1+ex_3)x_5\,.
\end{align}
Let us assume for the moment that $g \ne 0$, so that we can rescale $x_5$ and assume $g=1$. We make the change $x'_5=x_5+cx_1+fx_3$, rename and obtain
\[
\f_3=ax_1x_3+x_5(x_2+x_4)+(bx_1+ex_3)x_5=ax_1x_3+(bx_1-x_2+ex_3-x_4)x_5\,.
\]
The further change $x'_2=bx_1-x_2$, $x'_4=ex_3-x_4$, followed by adequately rescaling $x_3$ and $\f_2$, yields the desired model. Let us shows that it must indeed be $g \ne 0$ in \eqref{eq:5:3:1}. If it was $g=0$, then
\[
\f_3=ax_1x_3+x_1(c(x_2+x_4)+bx_5)+x_3(f(x_2+x_4)+ex_5)
\]
and $\f_3^2=-2(ce-bf)x_1x_3(x_2+x_4)x_5$, so we must have $ce-bf \ne 0$, so $(c,f) \ne (0,0) \ne (b,e)$. By swapping $x_1, x_3$ if necessary (and $x_2,x_4$ consequently), we can assume that $c \ne 0 \ne e$, and rescaling $x_2,x_4$ and $x_5$ we can assume $c=e=1$, so that $\f_3=ax_1x_3+x_1(x_2+x_4+bx_5)+x_3(f(x_2+x_4)+x_5)$. Consider now a generic third generator for $\pi$ of the form
\[
\f'_3=\f_3 + \a \f_1+\b\f_2 =x_1((1+\a)x_2+x_4+bx_5)+x_3(fx_2+(f+\b)x_4+x_5)+ax_1x_3\,.
\]
Imposing $(\f'_3)^2\neq 0$ we see that at least one of the coefficients 
\[
\b+\a f+\a\b\,, \quad 1+\a-bf\,, \quad 1-bf-b\b
\]
must be non-zero. Then $U'_3$ is generated by
\[
\begin{cases}
    (1+\a)x_2+x_4+bx_5+ax_3 \\
    (1+\a)x_1+fx_3 \\
    fx_2+(f+\b)x_4+x_5-ax_1\\
    x_1+(f+\b)x_3 \\
    bx_1+x_3
\end{cases}
\]
We see that $x_1,x_3 \in U'_3$, hence $U'_3=\la x_1,x_3,(1+\a)x_2+x_4+bx_5,fx_2+(f+\b)x_4+x_5\ra$. Therefore
\[
(1+\a)x_2+x_4+bx_5-b(fx_2+(f+\b)x_4+x_5)= (1+\a-bf)x_2+(1-bf-\b)x_4 \in U'_3\, .
\]
But if we take now $\a=bf-1 \ne -\b$, we deduce that $U_1\subset U'_3$, a contradiction.



\noindent As for the scheme-theoretic nature of $X$, the model shows that 
\[
X=\mathrm{Proj}(\kk[x,y,z]/(xy,z^2,xz,yz))\,,
\]
with Hilbert function $h(1)=3$ and $h(n)=2$ for $n \ge 2$. Hence $X$ is 0-dimensional and of degree $2$, and it consists of two simple points. 

\medskip

\noindent To finish, we deal with the case $|X|=1$. Put $X=\{\f_1\}$ and let $\f_2, \f_3$ denote points of $\pi$ which, together with $\f_1$, generate $\pi$. The choice of $\f_1$ is canonical up to rescaling, but the generators $\f_2, \f_3$ can be changed. Notice that $U_2 \cap U_3$ must have dimension $3$. Indeed, $\dim(U_2\cap U_3)\geq 3$ and if it were $U_2=U_3=U$ with $\dim U=4$, then the line $\PP(\la \f_2, \f_3 \ra)$ would intersect the Klein quadric $\GG \cap \PP(\L^2 U)$, so $X$ would have more than one point, contradicting our assumption.
Also, $\dim(U_1\cap U_i)\geq 1$ for $i=2,3$. We have further subcases, according to whether two, one, or none in $\{\f_2, \f_3\}$ have associated vector space containing $U_1$.

\noindent \textbf{Case 1.} $\pi$ has two generators $\f_2, \f_3$ such that $U_2$ and $U_3$ contain $U_{\f_1}$. We obtain simple generators in the following lemma.

\begin{lemma}\label{lem:case5-3-13}
Assume $X=\{\f_1\}$ and $\pi$ is generated by $\f_1, \f_2, \f_3$ so that both $U_2$ and $U_3$ contain $U_1$. Then we can choose (maybe different) generators $\f_2, \f_3$ for $\pi$ and coordinates $x_i$ for $W$ so that
\[ 
\left\{\begin{array}{ccl}
 \f_1 & = & x_1x_2\\
 \f_2 & = & x_1x_3+x_2x_4\\
 \f_3 & = & x_1x_5+x_2x_3
\end{array}\right.
\]
\end{lemma}
\begin{proof}
As usual, we denote $\f_2, \f_3$ two rank-4 generators for $\pi$ that may change along the process, and $x_i$ coordinates for $W$ that may also change.
We may take initial coordinates so that $\f_1=x_1x_2$ and $U_2=\langle x_1,x_2,x_3,x_4\rangle$, so $\f_2$ has the form
\[
\f_2=x_1(ax_3+bx_4)+x_2(cx_3+ex_4)+fx_3x_4
\]
for some $a,b,c,e,f\in\kk$. We can assume that the term $x_1x_2$ does not appear in $\f_2$ by subtracting a multiple of $\f_1$. Since $\f_2$ has rank 4, $ae-bc\neq 0$, and we can consider the change of coordinates $x'_3=ax_3+bx_4$, $x'_4=cx_3+ex_4$. We relabel the coordinates so that $\f_2=x_1x_3+x_2x_4+fx_3x_4$.
If $f \ne 0$, then we could arrange $f=1$ by rescaling $x_3$, $x_1$ and $\f_1$. But then the linear combination $\f_2+ \f_1=x_1(x_2+x_3)+(x_2+x_3)x_4$ would have rank $2$, a contradiction. Hence $f=0$ and $\f_2=x_1x_3+x_2x_4$. 

\item We know that $U_2\cap U_3$ has dimension three, so we can assume (maybe permuting $x_3,x_4$ and $x_1,x_2$ if necessary) that $U_2 \cap U_3=\langle x_1,x_2,x_3+bx_4\rangle$ for some $b\in\kk$. Changing $x'_3=x_3+bx_4$ we get $\f_2=x_1x'_3+(x_2-bx_1)x_4$, so if $x'_2=x_2-bx_1$ we get $\f_1=x_1x'_2$, $\f_2=x_1x'_3+x'_2x_4$, and $U_2\cap U_3=\langle x_1,x'_2,x'_3\rangle$, so $U_3=\langle x_1,x_2,x'_3,x_5\rangle$. We relabel again and write $\f_1=x_1x_2$, $\f_2=x_1x_3+x_2x_4$ and
\[
\f_3=x_1(ax_3+bx_5)+x_2(cx_3+ex_5)+fx_3x_5
\]
for some $a,b,c,e,f$. Since $\f_3$ has rank four we have $ae-bc \ne 0$. By permuting $x_1,x_2$ if necessary (and $x_3,x_4$ consequently) we can assume that both $b$ and $c$ are non-zero. By rescaling the coordinates $x'_5=bx_5$, $x'_3=cx_3$, $x'_4=cx_4$, and setting $\f'_2=c\f_2$, we may assume that $b=c=1$. With the further change $x'_5=ax_3+x_5$ we obtain $\f_3=x_1x'_5+x_2((1-ae)x_3+ex'_5)+fx_3x'_5$ and since $1-ae \ne 0$ we can rescale $x_3$, $x_4$ and $\f_2$ so that
\begin{equation} \label{eq:f_3-cases-e}
    \f_3=x_2x_3+(x_1+ex_2+fx_3)x_5 \, .
\end{equation}
Now we distinguish cases according to the value of $f$. If $f=0$ we make a change $x'_1=x_1+dx_2$, so $\f_3=x_2x_3+x'_1x_5$, $\f_1=x'_1x_2$, and 
\[
\f_2=(x'_1-ex_2)x_3+x_2x_4=x'_1x_3+x_2(-ex_3+x_4)
\]
so the proof is finished by putting $x'_4=x_4-ex_3$, since $\f_1, \f_2$ and $\f_3$ are expressed as in our desired model. To finish, we show that $f \ne 0$ in \eqref{eq:f_3-cases-e} leads to a contradiction. Indeed, in that case we could consider the linear combination
\[
\f'_3=f\f_3-\f_1= x_2(x_1+ex_2+fx_2)+f(x_1+ex_2+fx_3)x_5
\]
which has rank two, a contradiction.
\end{proof}

\noindent We study $X$ as a scheme. We have $X=\mathrm{Proj}(\kk[x,y,z]/(y^2,z^2,yz))$, so set-theoretically $X$ is the point $[1:0:0]$. Scheme-theoretically, it has Hilbert polynomial $h(n)=3$, so $X$ has dimension $0$ and degree $3$. This is a model for a triple point in a plane. Let us compute the tangent space: take the affine chart $A=\{x \ne 0\}$, so $X|_A=\mathrm{Spec}\left(\kk[y,z]/(y^2,yz,z^2)\right)$. The cotangent space at $(0,0)$ is $\{a y + b z, (a,b) \in \kk^2\}$, which has dimension $2$. We see that $X$ is a triple point with infinitesimal information given by a plane of tangent directions.

\noindent \textbf{Case 2.} $\pi$ contains exactly one line $\langle \f_1, \f_2\rangle$ such that for any $\f_2$ generating it, $U_2$ contains $U_1$, and there exists a third generator $\f_3$ such that $\f_3\left(U_2^0\right)$ is a line contained in $U_1$. Recall that for $U \subset W$ a subspace, we denote $U^0 \subset W^*$ its annihilator. We obtain a model for this case in the following Lemma. 

\begin{lemma}\label{lem:case5-3-14}
Assume $X=\{\f_1\}$ and that the set of bivectors containing $U_1$ forms a line, say $\langle \f_1, \f_2 \rangle =\{\f \in \pi \mid U_1\subset U_\f\}$. Assume moreover that there is a third generator $\f_3$ of $\pi$ so that $0 \ne \f_3(U^0_2) \subset U_1$. Then there are coordinates for $W$, and a choice of generators $\f_1, \f_2, \f_3$ for $\pi$ so that
\[ 
\left\{\begin{array}{ccl}
 \f_1 & = & x_1x_2\\
 \f_2 & = & x_1x_3+x_2x_4\\
 \f_3 & = & x_1x_5+x_3x_4
\end{array}\right.
\]
\end{lemma}

\noindent Notice that the model above satisfies the condition of the lemma: any bivector of the form $\f=\a\f_1+\b\f_2 + \f_3$ satisfies $U_{\f} \cap U_1=\la x_1 \ra$, hence the bivectors containing $U_1$ form a line.
\begin{proof}
As in the first part in the proof for Lemma \ref{lem:case5-3-13}, we get initial coordinates so that $\f_1=x_1x_2$, $\f_2=x_1x_3+x_2x_4$. Take any third generator $\f_3$. Note that $U_{23}\coloneqq U_2\cap U_3$ has dimension three and $U_{13}\coloneqq U_1\cap U_3$ has dimension one. Let us see that, after a suitable change, $U_{13}=\la x_1 \ra$. Indeed, permuting the pairs $(x_1,x_3)$ and $(x_2,x_4)$ if necessary we can assume $U_1 \cap U_3=\la x_1+bx_2 \ra$, so make the change $x'_1=x_1+bx_2$ and $\f_2=x'_1x_3+x_2x'_4$ with $x'_4=x_4-bx_3$. Reset notation and start again.

\noindent We now arrange so that $x_3, x_4 \in U_3$. The affine line $x_3 + \la x_2 \ra$ must intersect $U_3$ in a point, so we find $a\in\kk$ such that $x_3+ax_2 \in U_{\f_3}$; define $x'_3=x_3+ax_2$. Analogously, do the change $x'_4=x_4+bx_2$ for suitable $b\in\kk$. The generator $\f_2$ changes to $\f_2=x_1x'_3+x_2x'_4-ax_1x_2$, so we consider $\f'_2=\f_2+a\f_1=x_1x'_3+x_2x_4$. Reset notation again and we have $\f_1=x_1x_2$, $\f_2=x_1x_3+x_2x_4$, and $U_3=\la x_1, x_3,x_4,x_5 \ra$, so
\[
\f_3=x_5(ax_1+bx_3+cx_4)+x_4(ex_1+fx_3)+gx_1x_3 \, .
\]
Now, for any bivector $\f$ in the line generated by $\f_1, \f_2$ we have $U_{\f}^0=\langle v_5 \rangle$, being $\{v_i\}$ the basis of $W^*$ dual to $\{x_i\}$. We are assuming that $0 \ne \f_3(U_{\f}^0) \subset U_1$, so we must have that $0 \ne \f_3(v_5) \in \langle x_1 \rangle$, i.e. $b=c=0$ and $a \ne 0$ so (rescaling $x_1$ and $x_2$) we can assume $a=-1$ and 
\[
\f_3=x_1x_5+x_4(ex_1+fx_3)+gx_1x_3=x_1(x_5+gx_3)+x_4(ex_1+fx_3)=x_1x'_5+x'_3x_4
\]
where we write $x'_5=x_5+gx_3$, $x'_3=-fx_3-ex_1$. Note that $f \ne 0$ since otherwise $\f_3$ has rank two. With this change, $\f_2=-\frac1f x_1x'_3+x_2x_4$, so we rescale $x'_1=-\frac1f x_1$, and then $x''_5=-f x'_5$ in order to get $\f_3=x'_1x''_5+x'_3x_4$. We also rescale the first generator $\f'_1=-\frac1f \f_1=x'_1x_2$, and we get the desired model.
\end{proof}

\noindent As a scheme, we have $X=\mathrm{Proj}(\kk[x,y,z]/(xz-y^2,yz,z^2))$. As a set, this is the point $\{[1:0:0]\}$. As a variety, we see that its Hilbert function is $h(n)=3$, hence $X$ has dimension $0$ and degree $3$, and it is a triple point. In the affine chart $A=\{x \ne 0\}$ we have 
\[
X|_A=\mathrm{Spec}\left(\kk[y,z]/(z-y^2,yz,z^2)\right) \cong \mathrm{Spec}\left(\kk[y]/(y^3)\right)
\]
hence the cotangent space at $[1:0:0]$ is $\{a y, a \in \kk\}$ and has dimension $1$. We see that $X$ is a triple point with infinitesimal information given by one tangent direction of multiplicity $2$ in the direction of the $y$-axis.




\begin{remark}
It is a well-known result (see \cite[II.3.2]{Eisenbud-Harris}) that the two models of a triple point for $X \subset \pi$ from Cases 1 and 2 above are the only two isomorphism classes of a triple point in a plane (over an algebraically closed field).

\end{remark}

\noindent \textbf{Case 3.} $\pi$ has exactly one line $\langle \f_1, \f_2\rangle$ such that for any $\f_2$ generating it, $U_2$ contains $U_1$, and 
for any third generator $\f_3$ it holds that $\f_3(U_2^0)$ is a line not contained in $U_1$. We obtain the model in the following Lemma.

\begin{lemma}\label{lem:case5-3-15}
Assume $X=\{\f_1\}$ and that the bivectors containing $U_1$ form a line, say $\langle \f_1, \f_2 \rangle =\{\f \in \pi \mid U_1\subset U_{\f}\}$. Assume moreover that for any third generator $\f_3$ of $\pi$ we have $\f_3(U^0_2) \cap U_1=0$. Then there are coordinates for $W$, and a choice of generators $\f_1, \f_2, \f_3$ for $\pi$ so that:
\[ 
\left\{\begin{array}{ccl}
 \f_1 & = & x_1x_2\\
 \f_2 & = & x_1x_3+x_2x_4\\
 \f_3 & = & x_1x_4+x_3x_5
\end{array}\right.
\]
\end{lemma}

\noindent The model above satisfies the condition: any bivector of the form $\f=\a\f_1+\b\f_2 + \g\f_3$ with $\g \ne 0$ satisfies $\dim U_{\f} \cap U_1=1$, hence the bivectors containing $U_1$ form the line $\la \f_1, \f_2 \ra$. Also, any third generator $\f'_3=\f_3+\a\f_2+\b \f_1$ satisfies $\f'_3(v_5)=x_3 \notin U_{\f_1}$, with $\la v_5 \ra=U_2^0$. As usual $\{v_i\}$ is the basis dual to $\{x_i\}$.
\begin{proof}
Note first $\dim(U_2 \cap U_3)=3$ and $\dim(U_1 \cap U_3)=1$. 
As in the proof of Lemma \ref{lem:case5-3-13} we take initial coordinates so that $\f_1=x_1x_2$ and $\f_2=x_1x_3+x_2x_4$. By the same argument as in the proof of Lemma \ref{lem:case5-3-14}, we arrange that $x_1, x_3, x_4 \in U_3$, so that $U_3=\la x_1,x_3,x_4,x_5\ra$ and
\[
\f_3=x_1(ax_3+bx_4+cx_5)+x_3(ex_4+fx_5)+gx_4x_5 \, .
\]
As $\f_3(U_2^0)=\f_3(\langle v_5 \rangle)=\langle cx_1+fx_3+gx_4 \rangle$, at least one of $f$ or $g$ are non-zero. By permuting the coordinates $x_3, x_4$ if necessary (and also $x_1,x_2$ so that $\f_1$ and $\f_2$ are preserved), we can assume that $f \ne 0$, so we can rescale $x_5$ and assume $f=1$. We make the change $x'_5=ex_4+x_5$, and reset notation so that
\[
\f_3=x_1(ax_3+bx_4+cx_5)+x_3x_5+gx_4x_5=x_1(ax_3+bx_4)+(cx_1+x_3+gx_4)x_5 \, .
\]
Consider the change $x'_3=cx_1+x_3+gx_4$, so 
\[
\f_3=x_1(ax'_3+(b-ag)x_4)+x'_3x_5\quad \tr{and} \quad \f_2=x_1x'_3+(-gx_1+x_2)x_4=x_1x'_3+x'_2x_4
\]
putting $x'_2=x_2-gx_1$. Reset again notation so that $\f_1, \f_2$ remain as we want, and $\f_3$ has the form
\[
\f_3=x_1(ax_3+bx_4)+x_3x_5=bx_1x_4+x_3(-ax_1+x_5)
\]
with $b \ne 0$, so we can assume $b=1$ rescaling $x_1$, $x_2$, $\f_1$ and $\f_2$. Put $x'_5=ax_3+x_4$, so $\f_3=x_1x_4+x_3'x_5$, and we are done.
\end{proof}

\noindent As a scheme, we have $X=\mathrm{Proj}(\kk[x,y,z]/(xz,yz,y^2,z^2))$. As a set, this is the point $\{[1:0:0]\}$. As a variety, we see that its Hilbert function is $h(1)=3$ and $h(n)=2$ for $n\geq 2$, hence $X$ has dimension $0$ and degree $2$, and it is a double point.

\noindent \textbf{Case 4.} The plane $\pi$ does not have any point, other than $\f_1$, that contains $U_1$. This means that $U_{\f} \cap U_1$ has dimension $1$ for any $\f \in \pi$, $\f \ne \f_1$. We start with a preliminary Lemma.



\begin{lemma} \label{lem:case16-not-same-line} 
Suppose $\pi =\la \f_1, \f_2, \f_3 \ra \subset \PP(\L^2 W)=\PP^9$ is a plane such that $\pi \cap \GG=\{[\f_1]\}$. Assume also that dim~$(U_{\f} \cap U_1)=1$ for any $\f \in \pi \setminus \{[\f_1]\}$. Then the lines $U_{\f} \cap U_1$ are not all the same. In other words, $\bigcap_{\f \in \pi} U_{\f}=\{0\}$. 

\noindent In particular, we can choose generators $\f_1, \f_2, \f_3$ so that $U_1=\la x_1, x_2 \ra$, $U_2 \cap U_1=\la x_1 \ra$, and $U_3 \cap U_1=\la x_2 \ra$.
\end{lemma}
\begin{proof}
Assume otherwise, i.e. that the lines $U_{\f} \cap U_1=\la x_1 \ra$ are all the same.
First we will simplify the expressions for the generators of $\pi$, and then we will derive a contradiction. 

\noindent We choose an initial basis so that $\f_1=x_1x_2$ and $U_2=\la x_1, x_3, x_4, x_5 \ra$. With the usual changes of basis we can arrange so that $\f_2=x_1x_3+x_4x_5$. Now take a third generator $\f_3$. Since $x_1 \in U_3$ and $\dim U_3 \cap U_2=3$, at least one of $x_3,x_4, x_5$ is not in $U_3$. 
Let us assume that $x_4$ or $x_5 \notin U_3$ (in the case that $x_3$ is not in $U_3$ an analogous argument applies). Permuting $x_4, x_5$ if necessary we can assume it is $x_5$. The affine lines $x_3 + \la x_5 \ra$ and $x_4+ \la x_5 \ra$ intersect $U_{\f_3}$, so we can make changes $x'_3=x_3+ax_5$ and $x'_4=x_4+bx_5$ so that $x'_3, x'_4 \in U_{\f_3}$, and 
\[
\f_2=x_1(x'_3-ax_5)+x'_4x_5=x_1x'_3+(x'_4-ax_1)x_5=x_1x'_3+x''_4x_5 \, .
\]
Reset notation, and now we have that $\f_2=x_1x_3+x_4x_5$, $U_{\f_3}=\la x_1, x_3,x_4, x_2+x_5\ra$. A general third generator $\f'_3=\f_3+\a \f_1+\b \f_2$ and its square have the form
\begin{align*}
\f'_3&=x_1((a+\a)x_2+ax_5+(d+\b)x_3)+(cx_2+(c-\b)x_5+fx_3)x_4+ex_1x_4 \\
(\f'_3)^2&=-x_1x_4 
\left( 
\begin{vmatrix}
    a+\a & a \\
    c & c-\b 
\end{vmatrix}  
x_2x_5 + 
\begin{vmatrix}
    a+\a & d+\b \\
    c & f 
\end{vmatrix}  
x_2x_3 + 
\begin{vmatrix}
    a & d+\b \\
    c-\b & f 
\end{vmatrix}  
x_5x_3
\right)
\end{align*}
We get a contradiction if the coefficient of $x_5x_3$ in $(\f'_3)^2$ vanishes for some value of $\b$. In this case, we would have $\la x_1, x_2 \ra =U_{\f_1} \subset U_{\f'_3}$.
This coefficient vanishes for any $\b$ such that $\b^2+(d-c)\b +af-cd=0$, and this has some solution as $\kk$ is algebraically closed.
\end{proof}

\noindent In the next lemma we show how to control the plane $\pi$ in Case 4. Somewhat surprisingly, the key tool is a rational map of degree two.

\begin{lemma} \label{lem:rational-map}
Suppose $\pi =\la \f_1, \f_2, \f_3 \ra$ is a plane as in Lemma \ref{lem:case16-not-same-line}. Then, the map $\PP^1 \to \PP(U_1) \cong \PP^1$ given by
\[
[\a:\b] \mapsto U_{\a\f_2+\b\f_3} \cap U_1
\]
is a rational map of degree 2. With respect to a suitable choice of generators $\f_2, \f_3$ and basis $x_1,x_2$ of $U_1$, the map is given by the matrix:
\[
[\a : \b] \mapsto 
\begin{pmatrix}
    1 & a & 0 \\
    0 & b & 1
\end{pmatrix} 
\begin{pmatrix}
    \a^2 \\
    \a \b \\
    \b^2
\end{pmatrix}
\]
with $ab \ne 1$.
\end{lemma}
\begin{proof}
By the previous lemma we know the above map is non-constant. We can choose a basis $\{x_i\}$ and generators $\f_1, \f_2, \f_3$ for $\pi$ so that $U_1=\la x_1, x_2 \ra$, $U_2=\la x_1, x_3,x_4,x_5 \ra$ and $\la x_2 \ra= U_1 \cap U_3$. Then 
\[
\f_2=x_1(ax_3+bx_4+cx_5)+x_3(ex_4+fx_5)+gx_4x_5\,,
\]
and at least one of $a, b,c$ is non-zero. We can assume (swapping coordinates maybe) that $a \ne 0$, and rescaling $x_3$ we get $a=1$. Make the change $x'_3=x_3+bx_4+cx_5$, then $\f_2=x_1x'_3+x'_3(ex_4+fx_5)+g'x_4x_5$, with $g' \ne 0$, so rescaling $x_5$ we get $g'=1$. Reset notation, so $\f_2=x_1x_3+x_4(x_5-ex_3)+fx_3x_5$. Make the change $x'_5=x_5-ex_3$ so $\f_2=x_1x_3+x_4x'_5+fx_3x'_5$, reset notation so $\f_2=x_1x_3+(x_4+fx_3)x_5$, and put $x'_4=x_4+fx_3$. 

\item We start with $\f_1=x_1x_2$, $\f_2=x_1x_3+x_4x_5$ and $U_1\cap U_3=\la x_2 \ra$. For $j=3,4,5$ the affine line $x_j + \la x_1 \ra$ intersects $U_3$ in a point, hence we find $a_j \in \kk$ so that $x_j+a_jx_1 \in U_3$. In other words, through the change $x'_j=x_j+a_jx_1$, we can assume that $x_j \in U_3$, for $j=3,4,5$. Make these changes, and $\f_2=x_1x'_3+(x'_4-a_4x_1)(x'_5-a_5x_1)=x_1(x'_3+a_5x'_4\textcolor{red}{-}a_4x'_5)+x'_4x'_5$, so by a further change $x''_3=x'_3+a_5x'_4-a_4x'_5 \in U_3$, followed by relabeling, we get $\f_2=x_1x_3+x_4x_5$ and $U_3=\la x_2,x_3,x_4, x_5 \ra$. We have arrived to the following:
\[
\begin{cases}
    \f_1=x_1x_2 \\
    \f_2=x_1x_3+x_4x_5 \\
    \f_3=x_2(ax_3+bx_4+cx_5)+x_3(ex_4+fx_5)+gx_4x_5
\end{cases}
\]
If $b=c=0$ then $\f_3=(ax_2-dx_4-ex_5)x_3+fx_4x_5$ and $\f_3-f\f_2$ would have rank $2$, a contradiction. Hence $b$ or $c \ne 0$. Swapping the coordinates $x_4$ and $x_5$ (and maybe changing the sign of $\f_2$) we can assume that $c \ne 0$, so rescaling $x'_5=cx_5$, $x'_4=\frac1c x_4$ we can assume that $c=1$, and make the change $x'_5=x_5+bx_4$, so we can assume that $\f_3=x_2(ax_3+x_5)+x_3(ex_4+fx_5)+gx_4x_5$.

\item We claim now that $e\ne 0$. If it were $e=0$ then $\f_3=ax_2x_3+(x_2+fx_3+gx_4)x_5$, with $g \ne 0$, for otherwise the rank would drop. Then, the 4-space associated to the bivector $g\f_2-\f_3$ contains $U_1$, a contradiction.
We conclude that $e \ne 0$, so we can assume $e=1$ by rescaling $x'_4=ex_4$, $x'_1=ex_1$, and we get 
\[
\f_3=x_2(ax_3+x_5)+x_3(x_4+fx_5)+gx_4x_5=x_2(ax_3+x_5)-x_3x'_4+gx'_4x_5
\]
with a further change $x'_4=-(x_4+fx_5)$, $x'_5=-x_5$. Hence we arrive at a considerably simplified model
\begin{equation} \label{eq:model16-preliminary}
 \begin{cases}
    \f_1=x_1x_2 \\
    \f_2=x_1x_3+x_4x_5 \\
    \f_3=(ax_2+x_4)x_3+(x_2+bx_4)x_5
\end{cases}   
\end{equation}
where we have relabeled the constants $a,b$, with $ab \ne 1$ since $\f_3$ has rank $4$. Let us compute the map $[\a:\b] \mapsto U_{\a \b } \cap U_1$, with $U_{\a \b}\coloneqq U_{\a \f_2+ \b \f_3}$. We have $\a \f_2+\b \f_3=(\a x_1+a\b x_2+\b x_4)x_3+(\b x_2+(\a+b\b)x_4)x_5$, so it follows easily that
\[
U_{\a \b}=\la\a x_3,\b ax_3+\b x_5,\a x_1+\b a x_2+\b x_4,\b x_3 + (\a + b \b) x_5,\b x_2+(\a+b \b) x_4\ra
\]
and for $\a\b\ne 0$ this is generated by $x_3$,$x_5$, $y_1=\a x_1+\b ax_2+\b x_4$, and $y_2=\b x_2+(\a+b \b)x_4$. We eliminate $x_4$ by considering 
\[
(\a +b \b)y_1-\b y_2=(\a^2+b \a \b)x_1+(\tfrac{a}{ab-1} \a \b +\b^2) (ab-1)x_2
\]
with $ab-1 \ne 0$. With respect to the basis $x'_1=x_1$, and $x'_2=(ab-1)x_2$ of $U_1$, we have obtained the degree two rational map 
\[
(\a: \b) \mapsto 
\begin{pmatrix}
    \a^2+b\a \b \\
    \tfrac{a}{ab-1} \a \b+\b^2
\end{pmatrix}
=
\begin{pmatrix}
    1 & b & 0 \\
    0 & \tfrac{a}{ab-1} & 1
\end{pmatrix} 
\begin{pmatrix}
    \a^2 \\
    \a \b \\
    \b^2
\end{pmatrix}
\]
with $\tfrac{ab}{ab-1} \ne 1$, as desired.
\end{proof}

\noindent We shall use the rational maps from Lemma \ref{lem:rational-map} to get simplified generators for $\pi$ in the following manner. Fix a basis $x_1, x_2$ for the plane $U_1$ and take some generators $\f_2$, $\f_3$ of rank 4. The rational maps $[\a:\b] \mapsto U_{\a\f_2+\b \f_3} \cap U_1$ are well defined up to:
\begin{itemize}
    \item a change of the rank-4 generators $\f_2, \f_3$ of type
    \[
    \begin{pmatrix}
        \f_2' \\
        \f_3'
    \end{pmatrix} 
    =
    \begin{pmatrix}
        a & b \\
        c & d
    \end{pmatrix}
    \begin{pmatrix}
        \f_2 \\
        \f_3
    \end{pmatrix}
    \]
    which induces a change in the parameters $\a, \b$ so that $\a \f_2+\b \f_3=\a' \f_2' + \b' \f_3'$;
    \medskip
    \item a change of basis of $U_{\f}=\la x_1, x_2 \ra$. 
\end{itemize}  
In other words, the equivalence class of the rational maps from Lemma \ref{lem:rational-map} modulo these changes of bases is an invariant of the minimal algebra. It is easy to classify these equivalence classes in our (algebraically closed) field $\kk$.

\begin{lemma}\label{lem:rational-maps-equivalence-classes}
Let $\kk$ be algebraically closed. Let $U$, $V$ be 2-dimensional $\kk$-vector spaces. Any degree-2 rational map $\PP^1_{\kk}\cong \PP(U) \to \PP(V) \cong \PP^1_{\kk}$ of type
\[
[\a : \b] \mapsto 
\begin{pmatrix}
    1 & a & 0 \\
    0 & b & 1
\end{pmatrix} 
\begin{pmatrix}
    \a^2 \\
    \a \b \\
    \b^2
\end{pmatrix} , \quad
\text{ such that }
ab \ne 1.
\]
is equivalent, up to linear change of coordinates in both $U$ and $V$, to the map $[\a:\b] \to [\a^2: \b^2]$.
\end{lemma}
\begin{proof}
We denote such a rational map as
\[
\begin{cases}
    u=\a^2 + a \a \b  \\
    v=b \a \b + \b^2
\end{cases}
\]
and note that if $a=b=0$ there is nothing to prove. Let us assume $a \ne 0$ (the case $b \ne 0$ is analogous). Make the change $\a'=\a$, $\b'=a \b + \a$, so that
\[
\begin{cases}
    u'=u=\a' \b'  \\
    v'=a^2v= \b'^2+(1-ba) \a'^2+(ba-2)\a' \b'
\end{cases}
\implies 
\begin{cases}
    u''=u'=\a' \b'  \\
    v''=v'-(ba-2)u'= \b'^2+(1-ba) \a'^2
\end{cases}
\]
so we make the change:
\[
\begin{cases}
    \a''=\sqrt{1-ba} \, \a'  \\
    \b''=\b'
\end{cases} 
\, , \, 
\begin{cases}
    u'''=\sqrt{1-ba} \, u'' \\
     v'''=v''
\end{cases}
\]
and, if we reset notation $\a''=\a$, $u'''=u$, etc, we get to
\[
\begin{cases}
    u=\a \b  \\
    v=\a^2 +  \b^2
\end{cases}
\]
and then we are finished by putting
\[
\begin{cases}
    u'=v+2u=(\a + \b)^2=\a'^2  \\
    v'=v-2u=(\a-\b)^2=\b'^2 \, .
\end{cases}
\]
\end{proof}

\noindent Using Lemma \ref{lem:rational-maps-equivalence-classes} we can tackle our object of interest. 

\begin{proposition}\label{prop:case5-3-16}
Assume $\pi \cap \GG=\{\f_1\}$ and that for any $\f \in \pi \setminus \{\f_1\}$, $U_{\f} \cap U_1$ has dimension $1$. Then there are coordinates for $W$, and a choice of generators $\f_1, \f_2, \f_3$ for $\pi$ so that:
\[ 
\left\{\begin{array}{ccl}
 \f_1 & = & x_1x_2\\
 \f_2 & = & x_1x_3+x_4x_5\\
 \f_3 & = & x_3x_4+x_2x_5
\end{array}\right.
\]
\end{proposition}
\begin{proof}
In equation \eqref{eq:model16-preliminary} we showed that in a suitable basis we have
\[
\begin{cases}
    \f_1=x_1x_2 \\
    \f_2=x_1x_3+x_4x_5 \\
    \f_3=(ax_2+x_4)x_3+(x_2+bx_4)x_5
\end{cases}
\]
with $ab \ne 1$, since $\f_3$ has rank $4$. We are going to compute the rational map associated to this model, and transform this map to the canonical form $[\a^2:\b^2]$. The changes of basis necessary for doing this will lead us in the right direction to get our desired model.

\noindent \textbf{Case 1.} If $a=b=0$, by changing $x'_2=-x_2$, $x'_4=-x_4$, $x'_5=-x_5$ we get to promised model.

\medskip 

\noindent \textbf{Case 2.} If $b=0$, $a \ne 0$, by rescaling $x'_2=a x_2$, $x'_5=\frac1ax_5$, $x'_1=\frac1ax_1$ we can assume $a=1$, so we have 
\[
\begin{cases}
    \f_1=x_1x_2 \\
    \f_2=x_1x_3+x_4x_5 \\
    \f_3=(x_2+x_4)x_3+x_2x_5 \, .
\end{cases} 
\] 
We compute the rational map $U_{\a \f_2 + \b \f_3 } \cap U_1$. We get easily that $U_{\a \f_2 + \b \f_3 }=\la x_3,x_5,\a x_1+\b x_2+\b x_4,\b x_2+\a x_4\ra$. We eliminate $x_4$ from the third and fourth vectors and get the map:
\[
[\a:\b] \mapsto \a^2 x_1 + (\a \b -\b^2)x_2 = ux_1+vx_2 
\iff \begin{cases}
    u=\a^2 \\
    v=\a \b -\b^2
\end{cases}
\]
An easy computation shows that this map is equivalent to the rational map of giving the desired model:
\[
\begin{cases}
    u'=\a'^2 \\
    v'=\b'^2
\end{cases} \text{ with } 
\begin{cases}
    \a'=\a \\
    \b'=\a-2\b
\end{cases} \text{, }
\begin{cases}
    u'=u \\
    v'=u-4v
\end{cases}
\]
so the generators $\f'_2$ and $\f'_3$ are given by 
$\a \f_2+\b \f_3=\a' \f_2 + \tfrac{1}{2}(\a'-\b') \f_3= \a'(\f_2+\tfrac12 \f_3)-\tfrac12 \b' \f_3$,
which gives
\[
\begin{cases}
\f'_2= \f_2+\frac12 \f_3=x_1x_3+x_4x_5+\frac12 (x_2+x_4)x_3+\frac12 x_2x_5 \\
\f'_3=-\frac12 \f_3=-\frac12 (x_2+x_4) x_3-\frac12 x_2x_5
\end{cases}
\]and the change of basis in $U_1$ is given by $ux_1+vx_2=u'x_1+\tfrac14(u'-v')x_2=u'(x_1+\tfrac14 x_2)-\tfrac14 v'x_2$, which gives
\[
\begin{cases}
x'_1= x_1+\tfrac14 x_2 \\
x'_2=-\tfrac14 x_2
\end{cases} 
\iff
\begin{cases}
x_1= x'_1+ x'_2 \\
x_2=-4 x'_2
\end{cases} 
\]
Plugging these into $\f'_2$ and $\f'_3$ we get
\[
\begin{cases}
\f'_2= (x'_1+x'_2)x_3+x_4x_5+ (\frac12 x_4-2x'_2)x_3-2 x'_2x_5 =x'_1x_3 + (\frac12 x_4-x'_2)(x_3+2x_5)\\
\f'_3=(2x'_2-\frac12 x_4) x_3 + 2 x'_2x_5
\end{cases}
\]
so we introduce the change
\[
\begin{cases}
    x'_3=x_3 \\
    x'_4=\tfrac12 x_4 -x'_2 \\
    x'_5=x_3+2x_5
\end{cases} \text{ and we get:  }
\begin{cases}
\f'_2 =x'_1x'_3 + x'_4x'_5\\
\f'_3=(x'_2-x'_4) x_3 + 2 x'_2x_5=x'_2 x'_5 + x'_3x'_4
\end{cases}
\]
and this is the desired model. Note that $\f_1=x_1x_2$ is proportional to $x'_1x'_2$, and this is always ensured since the change of coordinates satisfies $\la x_1, x_2 \ra=\la x'_1, x'_2 \ra$ by construction.



\medskip

\noindent \textbf{Case 3.} If $b \ne 0$, $a=0$, by rescaling $x'_4=b x_4$, $x'_3=\frac1bx_3$, $x'_1=b^2x_1$ we can assume $b=1$, so we have 
\[
\begin{cases}
    \f_1=x_1x_2 \\
    \f_2=x_1x_3+x_4x_5 \\
    \f_3=x_4x_3+(x_2+x_4)x_5 \, .
\end{cases} 
\] 
An analogous computation as before yields a rational map $[\a: \b] \mapsto u x_1 + v x_2$ with
\[
\begin{cases}
    u= (\a+\b)\a \\
    v=-\b^2
\end{cases} 
\iff 
\begin{cases}
    u'=4u-v= (2\a+\b)^2=\a'^2 \\
    v'=-v=\b^2=\b'^2
\end{cases}
\]
so the right generators $\f'_2, \f'_3$ and basis $x'_1, x'_2$ are
\[
\begin{cases}
    \f'_2= \tfrac12 \f_2=\tfrac12 x_1x_3+\tfrac12 x_4x_5 \\
    \f'_3=-\tfrac12 \f_2+\f_3=(x_4-\tfrac12 x_1)x_3+(x_2+\tfrac12 x_4)x_5
\end{cases} \, ; \,
\begin{cases}
    x'_1= \tfrac14 x_1 \\
    x'_2=-\tfrac14 x_1-x_2
\end{cases}
\]
introducing the new coordinates $x'_1$, $x'_2$ in $\f'_2$, $\f'_3$ we get
\[
\begin{cases}
    \f'_2= x'_1(2x_3+x_5)+\tfrac12 (x_4-2x'_1)x_5=2x'_1x'_3+\tfrac12 x'_4x'_5 \\
    \f'_3=(x_4-2x'_1)(x_3+\tfrac12 x_5)-x'_2x_5=x'_4x'_3-x'_2x'_5
\end{cases}
\]
with the further change $x'_3=x_3+\tfrac12 x_5$, $x'_4=x_4-2x'_1$, $x'_5=x_5$.
Now it only remains to rescale $\f''_3=-\f'_3$ and $x'_1=\frac14 x''_1$ to get the desired model.

\medskip

\noindent \textbf{Case 4.} $ab\ne 0$. By rescaling $x'_2=a x_2$, $x'_5=\frac1a x_5$, $x'_1=\frac1ax_1$ we can assume $a=1$, so we have
\[
\begin{cases}
    \f_1=x_1x_2 \\
    \f_2=x_1x_3+x_4x_5 \\
    \f_3=(x_2+x_4)x_3+(x_2+bx_4)x_5
\end{cases}
\]
with $0 \ne b \ne 1$. Obviously, in this case we cannot rescale also $b$. The rational map is $ux_1+vx_2$ with 
\[
\begin{cases}
    u=\a (\a + b \b) \\
    v=\b(\a+(b-1)\b)
\end{cases}
\]
which is equivalent to $[u':v']=[\a'^2:\b'^2]$ with the change
\[
\begin{cases}
    u'=(1-2h+2 \ii \sqrt{h}\sqrt{1-h})u + b v \\
    v'=(1-2h-2 \ii \sqrt{h}\sqrt{1-h})u + b v
\end{cases} 
; \quad
\begin{cases}
    \a'=( \ii \sqrt{h} + \sqrt{1-h})\a + b\sqrt{1-h} \, \b \\
    \b'=( \ii \sqrt{h} - \sqrt{1-h})\a - b\sqrt{1-h} \, \b
\end{cases} 
\]
where we have denoted $h=\frac1b$, and $\ii=\sqrt{-1} \in \kk$ a choice for square root of $-1$. We need the inverse change, and this is given by
\[
\begin{cases}
    u=\tfrac{-\ii \sqrt{b}}{4 \sqrt{1-h}}(u'-v')  \\ 
    v=\left(\frac h2 + \ii \tfrac{(1-2h)\sqrt{h}}{4 \sqrt{1-h}}\right)u' + \left(\frac h2 - \ii \tfrac{(1-2h)\sqrt{h}}{4 \sqrt{1-h}}\right) v'
\end{cases} 
; \quad
\begin{cases}
    \a=\tfrac{-\ii}{2} \sqrt{b}(\a'+\b') \\
    \b=(\frac{h}{2\sqrt{1-h}} + \frac{\ii}{2} \sqrt{h})\a' + (\frac{-h}{2\sqrt{1-h}} + \frac{\ii}{2} \sqrt{h}) \b'
\end{cases} 
\]
From the relation $\a \f_2+ \b \f_3= \a' \f'_2 + \b' \f'_3$ we obtain the right second and third generators:
\[
\begin{cases}
    \f'_2=-\tfrac{\ii}{2}  \sqrt{b} \, \f_2 + \left(\frac{h}{2\sqrt{1-h}} + \frac{\ii}{2} \sqrt{h}\right) \f_3
    \\
    \f'_3=-\tfrac{\ii}{2}  \sqrt{b} \, \f_2 + \left(-\frac{h}{2\sqrt{1-h}} + \frac{\ii}{2} \sqrt{h} \right) \f_3
\end{cases} 
\]
and from the relation $ux_1+vx_2=u'x'_1+v'x'_2$, substituting $u',v'$ in terms of $u,v$, we get
\[
\begin{cases}
    x_1=(1-2h + 2 \ii \sqrt{1-h}\sqrt{h}) x'_1 + ( 1-2h-2 \ii \sqrt{1-h}\sqrt{h}) x'_2
    \\
    x_2=b x'_1 + bx'_2
\end{cases} 
\]
Now we substitute the expression of $\f_2$, $\f_3$ to get $\f'_2, \f'_3$ in terms of the $x_1,\dots , x_5$ basis, and then substitute $x_1, x_2$ in terms of $x'_1, x'_2$, and we obtain
\begin{align*}
    \f'_2 & =\left( (\sqrt{1-h} + \ii \sqrt{h} + \tfrac{1}{2 \sqrt{1-h}}) x'_1 + (-\sqrt{1-h} + \ii \sqrt{h} + \tfrac{1}{2 \sqrt{1-h}})x'_2 + (\tfrac{\ii}{2} \sqrt{h} + \tfrac{h}{2 \sqrt{1-h}})x_4   \right) x_3
    \\
    & + \left((\tfrac{\ii}{2} \sqrt{b} + \tfrac{1}{2 \sqrt{1-h}})(x'_1+x'_2) + \tfrac{1}{2 \sqrt{1-h}} x_4 \right) x_5 \\
    \f'_3 & =\left( (\sqrt{1-h} + \ii \sqrt{h} - \tfrac{1}{2 \sqrt{1-h}}) x'_1 + (-\sqrt{1-h} + \ii \sqrt{h} - \tfrac{1}{2 \sqrt{1-h}})x'_2 + (\tfrac{\ii}{2} \sqrt{h} - \tfrac{h}{2 \sqrt{1-h}})x_4   \right) x_3
    \\
    & + \left(\tfrac{\ii}{2} \sqrt{b} - \tfrac{1}{2 \sqrt{1-h}})(x'_1+x'_2) - \tfrac{1}{2 \sqrt{1-h}} x_4 \right) x_5
\end{align*} 
Now we make an ansatz $x_3=Ax'_3+Bx'_5$, $x_5=Cx'_3+Dx'_5$ for some constants to be determined. We impose that:
\begin{itemize}
    \item in $\f'_2$ the coefficients of $x'_2x'_3$ and $x_4x'_3$ vanish
    \item in $\f'_3$ the coefficients of $x'_1x'_5$ and $x_4x'_5$ vanish
\end{itemize}
 and get the (overdetermined) systems:
\[
\begin{cases}
    A((-\sqrt{1-h} + \ii \sqrt{h} + \tfrac{1}{2 \sqrt{1-h}})+ C ( \tfrac{\ii}{2} \sqrt{b} + \tfrac{1}{2 \sqrt{1-h}}) =0 \\
    A (\tfrac{\ii}{2} \sqrt{h} + \tfrac{h}{2 \sqrt{1-h}}) + C \tfrac{1}{2 \sqrt{1-h}} = 0
\end{cases}
\]
\[
\begin{cases}
    B((\sqrt{1-h} + \ii \sqrt{h} - \tfrac{1}{2 \sqrt{1-h}})+ D (\tfrac{\ii}{2} \sqrt{b} - \tfrac{1}{2 \sqrt{1-h}}) =0 \\
    B (\tfrac{\ii}{2} \sqrt{h} - \tfrac{h}{2 \sqrt{1-h}}) - D \tfrac{1}{2 \sqrt{1-h}} = 0
\end{cases}
\]
both of which have determinant zero, so they have parametric solutions
\[
\begin{cases}
    A=-\l \frac{1}{2\sqrt{1-h}}  \\
    C= \l (\frac{\ii}{2} \sqrt{h} +\frac{h}
    {2 \sqrt{1-h}}) \\
    B= \g \frac{1}{2 \sqrt{1-h}} \\
    D= \g (\frac{\ii}{2} \sqrt{h} - \frac{h}{2 \sqrt{1-h}})
\end{cases}
\text{ with } \l , \g \in \kk
\]
When we substitute this in the expressions of $\f'_2, \f'_3$ we get
\[
\begin{cases}
    \f'_2= - \l x'_1 x'_3 + \g \left( (\frac{1}{2} +\frac{\ii}{2} \frac{\sqrt{h}}{\sqrt{1-h}})x'_1 + (-\frac{1}{2} +\frac{\ii}{2} \frac{\sqrt{h}}{\sqrt{1-h}})x'_2 + \frac{\ii}{2} \frac{\sqrt{h}}{\sqrt{1-h}} x_4 \right) x'_5 \\
    \f'_3=\l \left( -(\frac{1}{2} +\frac{\ii}{2} \frac{\sqrt{h}}{\sqrt{1-h}})x'_1 + (\frac{1}{2} -\frac{\ii}{2} \frac{\sqrt{h}}{\sqrt{1-h}})x'_2 - \frac{\ii}{2} \frac{\sqrt{h}}{\sqrt{1-h}} x_4 \right) x'_3 - \g x'_2x'_5  
\end{cases}
\]
and some sort of a miracle allows us to define the last element of the basis
\[
x'_4=(\tfrac{1}{2} +\tfrac{\ii}{2} \tfrac{\sqrt{h}}{\sqrt{1-h}})x'_1 + (-\tfrac{1}{2} +\tfrac{\ii}{2} \tfrac{\sqrt{h}}{\sqrt{1-h}})x'_2 + \tfrac{\ii}{2} \tfrac{\sqrt{h}}{\sqrt{1-h}} x_4
\]
so that we have
\[
\begin{cases}
    \f'_2= - \l x'_1 x'_3 + \g x'_4 x'_5 \\
    \f'_3=-\l x'_4 x'_3 - \g x'_2x'_5  
\end{cases}
\]
so if we choose $\l=1$, $\g=-1$, and change the sign to $\f'_2$ we get
\[
\begin{cases}
    -\f'_2= x'_1 x'_3 + x'_4 x'_5 \\
    \f'_3= x'_3 x'_4 + x'_2x'_5  
\end{cases}
\]
and this our sought model.
\end{proof}

\noindent To conclude, we study $X$ scheme-theoretically. Plugging the parametric equations of $\pi$ into \eqref{eq:G(2,5)} yields
\[
X=\mathrm{Proj}(\kk[x,y,z]/(xz,yz,yz,y^2,z^2))\,.
\]
As expected, set-theoretically $X=\{[1:0:0]\}$. The Hilbert function of $X$ is $h(1)=3$ and $h(n)=1$ if $n \ge 2$, hence $X$ has dimension $0$ and degree $1$, so $X$ is a simple point also as a scheme.

\noindent The last case to consider is $X=\emptyset$. This is the generic case by dimension arguments, since $\dim \GG=6$, so a generic plane $\pi \subset \PP^9$ is disjoint from $\GG$. 
For instance, this case occurs if $\pi$ is generated by
\begin{equation} \label{eq:plane-case17}
\left\{\begin{array}{ccl}
 \f_1 & = & x_1x_2+x_3x_4\\
 \f_2 & = & x_1x_3+x_4x_5\\
 \f_3 & = & x_1x_5+x_2x_3
\end{array}\right.
\end{equation}
Indeed, a linear combination $\f=\a \f_1+\b\f_2+\g\f_3$ has square
\[
\f^2=\a^2x_1x_2x_3x_4+(\b^2+\a\g)x_1x_3x_4x_5+ \g^2 \, x_1x_2x_3x_5 +\a\b \, x_1x_2x_4x_5+\b \g \, x_2x_3x_4x_5
\]
which is non-zero unless $\a=\b=\g=0$.
We need to see that the above is the only model satisfying the condition $\pi \cap \GG=\emptyset$. 

\begin{lemma}
Under the action of $\PGL(W)$ in $\PP(\L^2 W)=\PP^9$, all planes $\pi \subset \PP^9$ with $\pi \cap \GG=\emptyset$ are in the same orbit, whose representative is given by \eqref{eq:plane-case17}.
\end{lemma}
\begin{proof}
This was proved in \cite[Proposition 1]{Manivel-Mezzetti}. The proof follows these lines: 
\begin{itemize}
    \item First, one sees that the orbit of the plane $\pi$ from equation \eqref{eq:plane-case17} is a Zariski-open set inside the Grassmannian $\GG(2,\PP^9)$ of projective planes of $\PP^{9}=\PP(\L^2 W)$, i.e. that 
    \[
    \mathcal{O}_{\pi}=\PGL(W) \cdot \pi \subset \GG(2,9)
    \]
    is Zariski-open. This is done by computing explicitly the dimension of the Lie algebra of the stabilizer of $\pi$. The details of the computation are in \cite[Section 2]{Manivel-Mezzetti}. This dimension turns out to be $3$, hence the orbit has dimension 
    \[
    \dim \mathcal{O}_{\pi}= \dim \PGL(W)-3=21= \dim \GG(2,9)
    \]
    so $\mathcal{O}_{\pi}$ is open.
    \medskip
    \item Then, one starts from a generic plane $\pi'$ satisfying $\pi' \cap \GG=\emptyset$, simplifies a bit the model of $\pi'$, and then computes the dimension of the stabilizer of $\pi'$, which is also $3$. This shows that the orbit of $\pi'$ is also open in $\mathrm{Gr}(2,9)$.
    \medskip
    \item Finally, two Zariski-open subsets of the variety $\GG(2,9)$ must intersect, as the Grassmannian is irreducible. Hence the orbits of $\pi$ and $\pi'$ intersect, so they coincide.
\end{itemize}
\end{proof}

\noindent The last 8 models are characterized by the property that $\pi\cap\GG$ is a finite (perhaps empty) set. We collect these results in Table \ref{table:5-3_3}.
\begin{itemize}
    \item The second column contains the scheme-theoretic intersection of $\pi$ and the Grassmannian $\GG=\GG(1,4)$;
    \item the third, fourth and fifth contain the differentials of the non-closed elements;
    \item all the minimal algebras appearing in this table are irreducible. The sixth column identifies our algebra with the Lie algebra in the list obtained in \cite{Ren2011}.
\end{itemize}

\begin{table}[h]
\caption{Minimal algebras of type $(5,3)$ with $\pi\cap\GG$ finite/empty}\label{table:5-3_3}
{\tabulinesep=1mm
\begin{tabu}{|c|c|c|c|c|c|c|c|c|c|c|c|c|c|}
\hline
Label & $\pi\cap\GG$ & $dx_6$ & $dx_7$ & $dx_8$ & \cite{Ren2011}\\
\hline
(5.3.10) & $\{p,q,r\}$ & $x_1x_2$ & $x_3x_4$ & $x_1x_5+x_3x_5$ & $N_{10}^{8,3}$ \\
\hline
(5.3.11) & $\{2p,q\}$ & $x_1x_2$ & $x_3x_4$ & $x_1x_3+x_2x_5$ & $N_{3}^{8,3}$\\
\hline
(5.3.12) & $\{p,q\}$ & $x_1x_2$ & $x_3x_4$ & $x_1x_3+(x_2+x_4)x_5$ & $N_{11}^{8,3}$\\
\hline
(5.3.13) & $\{3p\}+2 \tr{ dir}$ & $x_1x_2$ & $x_1x_3+x_2x_4$ & $x_1x_5+x_2x_3$ & $N_{8}^{8,3}$\\
\hline
(5.3.14) & $\{3p\}+1 \tr{ dir}$ & $x_1x_2$ & $x_1x_3+x_2x_4$ & $x_1x_5+x_3x_4$ & $N_{7}^{8,3}$\\
\hline
(5.3.15) & $\{2p\}$ & $x_1x_2$ & $x_1x_3+x_2x_4$ & $x_1x_4+x_3x_5$ & $N_{6}^{8,3}$\\
\hline
(5.3.16) & $\{2p\}$ & $x_1x_2$ & $x_1x_3+x_4x_5$ & $x_3x_4+x_2x_5$ & $N_{4}^{8,3}$\\
\hline
(5.3.17) & $\emptyset$ & $x_1x_2+x_3x_4$ & $x_1x_3+x_4x_5$ & $x_1x_5+x_2x_3$ & $N_{9}^{8,3}$\\
\hline
\end{tabu}}
\end{table}

\section{Case (4,4)}\label{sec:(4,4)}

\noindent We have $d\colon F_1 \to \L^2 W_0$ injective, with $\dim F_1=4$, and $\pi=\PP(d(F_1))$ is a projective $3$-plane in $\PP^5=\PP(\L^2 W)$, where $W=W_0$. As in the previous case, every bivector in $\L^2 W$ has rank at most $4$, and the rank stratification has one non-trivial stratum: the rank-$2$ bivectors given by the Plücker embedding of the Grassmannian $\tr{Gr}(2,4)$ of planes in $W \cong \kk^4$, or equivalently the Grassmannian $\GG(1,3)$ of projective lines in $\PP(W) \cong \PP^3$. We need to study relative positions of $\pi$ and $\GG(1,3)$ inside $\PP^5=\PP(\L^2 W)$. Let us set $\GG=\GG(1,3)$ in this section. It is well-known that the Plücker embedding sends $\GG$ to the Klein quadric, a smooth quadric in $\PP^5$. 

\subsection{Properties of quadrics} We start by recalling a few facts about quadrics; a reference for this is \cite[Chapter 22]{Harris}. Let $\cQ\subset\PP(V)=\PP^n
$ be a quadric, the vanishing locus of a homogeneous polynomial $Q$ of degree 2. Then $Q\colon V\times V\to\kk$ is a quadratic form. The {\em rank} of $\cQ$ is the rank of the linear map $\tilde{Q}\colon V\to V^*$, $\tilde{Q}(v)(w)=Q(v,w)$. $\cQ$ has maximal rank if and only if it is smooth.

\begin{lemma}\label{lem:quadric_intersection}
Let $\Lambda\cong\PP^{n-k}\subset\PP^n$ be a linear subspace and set $\cQ'=\cQ\cap \Lambda$. Then
\[
\textrm{rank}(\cQ)-2k\leq \textrm{rank}(\cQ')\leq \textrm{rank}(\cQ)\,.
\]
\end{lemma}
\begin{proof}
 Suppose $\Lambda=\PP(W)$ with $W\cong\kk^{n+1-k}$. Then $\rank(\cQ')$ is the rank of the linear map $\Tilde{Q}'\colon W\to W^*$, obtained as the composition of the following maps:
 \[
 W\xrightarrow{i} V\xrightarrow{\Tilde{Q}}V^*\xrightarrow{i^*}W^*\,.
 \]
 The inequality $\textrm{rank}(\cQ')\leq \textrm{rank}(\cQ)$ is obvious. For the second one, applying rank-nullity to the linear map $\Tilde{Q}\big|_W=\Tilde{Q}\circ i\colon W\to V^*$ we obtain
 \[
 \dim \Tilde{Q}(W)=\dim W-\dim\left(\ker \Tilde{Q}\big|_W\right)=\dim W-\dim(\ker\Tilde{Q}\cap W)\geq \dim W-\dim V+\rank(\cQ)\,,
 \]
 since $\ker\Tilde{Q}\cap W\subset\ker\Tilde{Q}$, hence $\dim(\ker\Tilde{Q}\cap W)\leq \dim\ker\Tilde{Q}=\dim V-\rank(\cQ)$. Consider next the linear map $i^*\big|_{\Tilde{Q}(W)}\colon\Tilde{Q}(W)\to W^*$; again by rank-nullity we have
 \begin{align*}
 \dim\left(\im i^*\big|_{\Tilde{Q}(W)}\right)&=\dim \Tilde{Q}(W)-\dim\left(\ker i^*\big|_{\Tilde{Q}(W)}\right)=\dim \Tilde{Q}(W)-\dim\left(\Tilde{Q}(W)\cap W^0\right)\\
 &\geq\dim \Tilde{Q}(W)-\dim W^0\,,    
 \end{align*}
 since $\Tilde{Q}(W)\cap W^0\subset W^0$, hence $\dim(\Tilde{Q}(W)\cap W^0)\leq \dim W^0$; here $W^0$ denotes the annihilator of $W$ in $V^*$. Altogether, we have
 
 \begin{align*}
 \rank(\cQ')&=\dim(\im \Tilde{Q}')=\dim \left(\im i^*\big|_{\Tilde{Q}(W)}\right)\geq\dim \Tilde{Q}(W)-\dim W^0\\
 &\geq \dim W-\dim V+\rank(\cQ)-\dim W^0=(n+1-k)-(n+1)+\rank(\cQ)-k\\
 &=\rank(\cQ)-2k\,.
 \end{align*}

\end{proof}

\subsection{Analysis of cases} Since the Klein quadric $\GG$ is smooth, it has rank 6. We are interested in the possible intersections of $\pi\cong\PP^3$ with $\GG$. By Lemma \ref{lem:quadric_intersection} the rank $r$ of $\pi\cap\GG$ satisfies $2\leq r\leq 4$. We call $\f_i$, $i=5,\ldots,8$ some generators of $\pi \subset \PP(\L^2 W)$ which we try to choose as simple as possible in suitable coordinates.

\noindent \textbf{Case 1: $r=4$.} In this case $\pi \cap \GG$ is a smooth quadric surface. A smooth quadric surface in $\pi=\PP^3$ is known to be a ruled surface $\mathcal{S}$: it contains two rulings of lines such that the lines of the first ruling are disjoint, as are the lines of the second ruling, and each line of the first ruling meets each line of the second ruling in one point. We choose our bivectors as follows:
\begin{itemize}
    \item let $\f_5$ be any point in $\mathcal{S}$ and denote by $\ell_1$ (resp. $\ell_2$) the line of the first (resp. second) ruling passing through $\f_5$;
    \item choose $\f_6$ on $\ell_1$ and $\f_7$ in $\ell_2$; denote by $\ell_3$ (resp. $\ell_4$) the other line, contained in $\mathcal{S}$, passing through $\f_6$ (resp. $\f_7$);
    \item set $\f_8\coloneqq \ell_3\cap\ell_4$.
\end{itemize}
It is easy to see that $\la\f_5,\ldots,\f_8\ra=\pi$ and that this choice of bivectors gives the following model:
\[
\begin{cases}
\f_5=x_1x_2 \\
\f_6=x_1x_3 \\
\f_7=x_2x_4 \\
\f_8=x_3x_4 \\
\end{cases}
\]

\noindent \textbf{Case 2: $r=3$.} In this case $\pi \cap \GG$ is a quadric cone; in other words, $\pi \cap \GG$ is the cone over a smooth conic $\mathscr{C}$ contained in some $\PP^2\subset\pi$. We choose our bivectors as follows:
\begin{itemize}
    \item $\f_5$ is the vertex of the cone;  
    \item pick $\f_6$ and $\f_7$ on $\mathscr{C}$ and consider the tangent lines $\ell_6$ and $\ell_7$ to $\mathscr{C}$ in $\PP^2$;
    \item set $\f_8\coloneqq \ell_6\cap\ell_7$.
\end{itemize}
Clearly $\la\f_5,\ldots,\f_8\ra=\pi$. Taking coordinates, it is easy to see that this choice of bivectors gives the model
\[
\begin{cases}
\f_5=x_1x_2 \\
\f_6=x_1x_4 \\
\f_7=x_2x_3 \\
\f_8=x_1x_3-x_2x_4 \\
\end{cases}
\]

\noindent \textbf{Case 3: $r=2$.} In this case $\pi \cap \GG=\pi_1\cup\pi_2$, that is, $\pi \cap \GG$ consists of two 2-planes. We need more facts about the Klein quadric. It is known (see \cite[Example 22.7]{Harris}) that the planes contained in $\GG$ form a 3-dimensional Fano variety with two irreducible connected components. Hence $\GG$ contains two rulings by 2-planes. We have the following result (see \cite[Proposition 22.8]{Harris}):

\begin{proposition}
Two 2-planes of the same ruling in $\GG$ either coincide or intersect in a single point; two 2-planes of opposite ruling either intersect in a line or are disjoint.
\end{proposition}

\noindent The planes of the first ruling consist of vector $2$-planes contained in some 3-dimensional subspace $H\subset W$, and the planes of the second ruling consist of vector $2$-planes containing a given line $r\subset W$. Since $\pi_1$ and $\pi_2$ are contained in $\pi$ they intersect in a line $\ell$, hence they belong to opposite rulings in $\GG$. We choose our bivectors as follows:
\begin{itemize}
    \item we pick $\f_5$ and $\f_6$ on $\ell$;
    \item we choose $\f_7$ in $\pi_1$, away from $\ell$;
    \item we choose $\f_8$ in $\pi_2$, away from $\ell$.
\end{itemize}
It is easy to see that $\la\f_5,\ldots,\f_8\ra=\pi$ and that this choice of bivectors gives the following model:
\[
\begin{cases}
\f_5=x_1x_2 \\
\f_6=x_1x_3 \\
\f_7=x_2x_3 \\
\f_8=x_1x_4 \\
\end{cases}
\]

\noindent We thus obtain 3 minimal algebras of type $(4,4)$ over $\kk$ (or any quadratically closed field). This agrees with \cite[Corollary 7.5]{Stroppel}. We collect these results in Table \ref{table:4-4}.
\begin{itemize}
    \item The second column contains the rank of the quadric obtained by intersecting $\pi\cong\PP^3$ with the Grassmannian $\GG=\GG(1,3)$;
    \item columns three to six contain the differentials of the non-closed elements;
    \item all the minimal algebras appearing in this table are irreducible. The sixth column identifies our algebra with the Lie algebra in the classification obtained in \cite{Yan}.
\end{itemize}

\begin{table}[h]
\caption{Minimal algebras of type $(4,4)$}\label{table:4-4}
{\tabulinesep=1mm
\begin{tabu}{|c|c|c|c|c|c|c|c|c|c|c|c|c|c|}
\hline
Label & $\mathrm{rank}(\pi \cap \GG)$ & $dx_5$ & $dx_6$ & $dx_7$ & $dx_8$ & \cite{Yan}\\
\hline
(4.4.1) & 4 & $x_1x_2$ & $x_1x_3$ & $x_2x_4$ & $x_3x_4$ & $N_1^{8,4}$ \\
\hline
(4.4.2) & 3 & $x_1x_2$ & $x_1x_4$ & $x_2x_3$ & $x_1x_3-x_2x_4$ & $N_{3}^{8,4}$\\
\hline
(4.4.3) & 2 & $x_1x_2$ & $x_1x_3$ & $x_2x_3$ & $x_1x_4$ & $N_2^{8,4}$\\
\hline
\end{tabu}}
\end{table}

\section{The complete list}

In this last section we include a table with all 2-step nilpotent 8-dimensional Lie algebras over an algebraically closed field $\kk$ of characteristic $\ne 2,3$. The fact that, in the irreducible case, our list coincides with other lists in the literature such as \cite{Ren2011,Yan} shows that, indeed, no specific properties of the complex numbers are needed, apart from being algebraically closed and of characteristic $\ne 2,3$. We include the dimension of the center (which is computed easily) and the Betti numbers, for which we used the package {\em Commutative Differential Graded Algebras} from \texttt{SageMath}. Clearly $b_1=\dim W_0$ and the remaining Betti numbers can be computed by Poincaré duality, since every nilpotent Lie algebra is unimodular.

\begin{remark}
We believe that this classification is still valid for fields of characteristic three; however, in this setting some technical difficulties appear in the $(6,2)$-case as the condition $\f^3=0$ is trivially satisfied, so the rank-3 bivectors must be characterized otherwise.
\end{remark}

\begin{table}[h!]\label{table:final}
\caption{8-dimensional 2-step nilpotent Lie algebras over $\kk$}
{\tabulinesep=1mm
\begin{tabu}{|c|c|c|c|c|c|c|c|c|c|c|c|c|c|}
\hline
$\frg$ & $dx_5$ & $dx_6$ & $dx_7$ & $dx_8$ & $\dim \mathfrak{z}(\frg)$ & $b_2$ & $b_3$ & $b_4$\\
\hline
(8.0.1) & 0 & 0 & 0 & 0 & 8 & 28 & 56 & 70\\
\hline
(7.1.1) & 0 & 0 & 0 & $x_1x_2$ & 6 & 22 & 41 & 50\\
\hline
(7.1.2) & 0 & 0 & 0 & $x_1x_2+x_3x_4$ & 4 & 20 & 33 & 38\\
\hline
(7.1.3) & 0 & 0 & 0 & $x_1x_2+x_3x_4+x_5x_6$ & 2 & 20 & 28 & 28\\
\hline
(6.2.1) & 0 & 0 & $x_1x_2$ & $x_1x_3$ & 5 & 18 & 34 & 42\\
\hline
(6.2.2) & 0 & 0 & $x_1x_2$ & $x_3x_4$ & 4 & 17 & 30 & 36\\
\hline
(6.2.3) & 0 & 0 & $x_1x_2$ & $x_1x_3+x_2x_4$ & 4 & 17 & 30 & 36\\
\hline
(6.2.4) & 0 & 0 & $x_1x_2$ & $x_1x_3+x_4x_5$ & 3 & 15 & 26 & 32\\
\hline
(6.2.5) & 0 & 0 & $x_1x_2+x_3x_4$ & $x_1x_3+x_2x_5$ & 3 & 14 & 24 & 30\\
\hline
(6.2.6) & 0 & 0 & $x_1x_2+x_3x_4$ & $x_1x_5+x_3x_6$ & 2 & 13 & 23 & 30\\
\hline
(6.2.7) & 0 & 0 & $x_1x_2+x_3x_4$ & $x_3x_4+x_5x_6$ & 2 & 13 & 22 & 28\\
\hline
(6.2.8) & 0 & 0 & $x_1x_2+x_3x_4$ & $x_1x_5+x_3x_6$ & 2 & 13 & 23 & 30\\
\hline
(6.2.9) & 0 & 0 & $x_1x_2+x_3x_4$ & $x_1x_5+x_2x_3+x_4x_6$ & 2 & 13 & 22 & 28\\
\hline
(6.2.10) & 0 & 0 & $x_1x_2$ & $x_3x_4+x_5x_6$ & 2 & 15 & 24 & 28\\
\hline
(6.2.11) & 0 & 0 & $x_1x_2$ & $x_1x_3+x_2x_4+x_5x_6$ & 2 & 15 & 24 & 28 \\
\hline
(5.3.1) & 0 & $x_1x_2$ & $x_1x_3$ & $x_2x_3$ & 5 & 15 & 31 & 40\\
\hline
(5.3.2) & 0 & $x_1x_2$ & $x_1x_3$ & $x_1x_4$ & 4 & 16 & 30 & 36 \\
\hline
(5.3.3) & 0 & $x_1x_2$ & $x_3x_4$ & $x_1x_3+x_2x_4$ & 4 & 15 & 25 & 28 \\
\hline
(5.3.4) & 0 & $x_1x_2$ & $x_1x_3$ & $x_3x_4$ & 4 & 15 & 27 & 32 \\
\hline
(5.3.5) & 0 & $x_1x_2$ & $x_1x_3$ & $x_1x_4+x_2x_3$ & 4 & 15 & 28 & 34 \\
\hline
(5.3.6) & 0 & $x_1x_2$ & $x_1x_3$ & $x_4x_5$ & 3 & 14 & 25 & 30\\
\hline
(5.3.7) & 0 & $x_1x_2$ & $x_1x_3$ & $x_1x_4+x_2x_5$ & 3 & 14 & 25 & 30\\
\hline
(5.3.8) & 0 & $x_1x_2$ & $x_1x_3$ & $x_2x_4+x_3x_5$ & 3 & 13 & 24 & 30\\
\hline
(5.3.9) & 0 & $x_1x_2$ & $x_1x_3$ & $x_2x_3+x_4x_5$ & 3 & 12 & 23 & 30\\
\hline
(5.3.10) & 0 & $x_1x_2$ & $x_3x_4$ & $x_1x_5+x_3x_5$ & 3 & 13 & 23 & 28\\
\hline
(5.3.11) & 0 & $x_1x_2$ & $x_3x_4$ & $x_1x_3+x_2x_5$ & 3 & 13 & 23 & 28\\
\hline
(5.3.12) & 0 & $x_1x_2$ & $x_3x_4$ & $x_1x_3+(x_2+x_4)x_5$ & 3 & 12 & 22 & 28\\
\hline
(5.3.13) & 0 & $x_1x_2$ & $x_1x_3+x_2x_4$ & $x_1x_5+x_2x_3$ & 3 & 13 & 24 & 30\\
\hline
(5.3.14) & 0 & $x_1x_2$ & $x_1x_3+x_2x_4$ & $x_1x_5+x_3x_4$ & 3 & 13 & 23 & 28\\
\hline
(5.3.15) & 0 & $x_1x_2$ & $x_1x_3+x_2x_4$ & $x_1x_4+x_3x_5$ & 3 & 12 & 22 & 28\\
\hline
(5.3.16) & 0 & $x_1x_2$ & $x_1x_3+x_4x_5$ & $x_3x_4+x_2x_5$ & 3 & 12 & 22 & 28\\
\hline
(5.3.17) & 0 & $x_1x_2+x_3x_4$ & $x_1x_3+x_4x_5$ & $x_1x_5+x_2x_3$ & 3 & 12 & 22 & 28\\
\hline
(4.4.1) & $x_1x_2$ & $x_1x_3$ & $x_2x_4$ & $x_3x_4$ & 4 & 14 & 25 & 28\\
\hline
(4.4.2) & $x_1x_2$ & $x_1x_4$ & $x_2x_3$ & $x_1x_3-x_2x_4$ & 4 & 14 & 25 & 28\\
\hline
(4.4.3) & $x_1x_2$ & $x_1x_3$ & $x_2x_3$ & $x_1x_4$ & 4 & 14 & 26 & 30\\
\hline
\end{tabu}}
\end{table}

\section*{Data availability statement}
No additional data were generated or analyzed in this research.

\section*{Conflict of interests}
On behalf of all authors, the corresponding author states that there is no conflict of interest.

\newpage

\bibliographystyle{plain}

\end{document}